\documentclass[11pt]{amsart}
\usepackage[utf8x]{inputenc}
\usepackage[T1]{fontenc}
\usepackage[top=3cm,bottom=2cm,left=3cm,right=3cm,marginparwidth=1.75cm]{geometry}

\usepackage{amsmath}
\usepackage{graphicx}
\usepackage{import}
\usepackage{amsthm}
\usepackage[colorlinks=true, allcolors=blue]{hyperref}
\usepackage{subcaption}
\usepackage{cite}
\theoremstyle{plain}
\newtheorem{thm}{Theorem}[section]
  \theoremstyle{definition}
  
  \theoremstyle{plain}
  \newtheorem{prop}[thm]{Proposition}
  \theoremstyle{plain}
  \newtheorem{lem}[thm]{Lemma}
  \theoremstyle{remark}
  \newtheorem{rem}[thm]{Remark}
\theoremstyle{remark}

\title{The "Hot Spots" Conjecture on the Vicsek Set}
\author{Marius Ionescu}
\address{Department of Mathematics\\  United States Naval Academy\\
  Annapolis, MD 21402 USA}

\email{ionescu@usna.edu}

\author{Thomas L. Savage}
\address{Department of Mathematics\\  United States Naval Academy\\
  Annapolis, MD 21402 USA}

\email{tsavage1352@att.net}

\begin{document}

\begin{abstract}
We prove the “hot spots” conjecture on the Vicsek set. Specifically, we
will show that every eigenfunction of the second smallest eigenvalue
of the Neumann Laplacian on the Vicsek set attains its maximum and
minimum on the boundary. 
\end{abstract}

\maketitle

\section{Introduction}

The ``hot spots'' conjecture studies whether a flat piece of metal that
is given an initial heat distribution will achieve its highest
temperatures on its boundary given enough time. That is, the
conjecture claims that in a
two-dimensional, bounded, connected domain $D$, the heat at point $x$
at time $t$, $u(x,t)$, achieves its maximum value on the boundary of
$D$. The “hot spots” conjecture was first posed by Rauch in 1974. An
equivalent formulation of the conjecture is as follows: every eigenfunction of
the second eigenvalue of the Neumann Laplacian attains its maximum and
minimum on the boundary. The conjecture has been shown to be true for some
Euclidean domains \cite{RoKry_JFA99,RaKry_JAMS04,DaNi_JAMS00,Ya_JMP09}
including recently for thin curved strips \cite{KrTu_2017}, but it has
also been shown to fail in others \cite{KryWe_AM99,Kry_Duke05}.

There is now a notion of a Laplacian on many fractals, and the theory
of eigenfunctions of the Laplacian is well developed in many
cases. Therefore, one can formulate the ``hot spot'' conjecture on
these fractals.  We are going to use the theory developed by Kigami
\cite{Kig}, see also \cite{Strichartz} that applies to the class of
post critically finite (p.c.f.) fractals. For many such fractals,
eigenvalues and eigenfunctions of the Laplacian can be computed
explicitly via a method called spectral decimation
\cite{Ra_Tou,Shi_JJIAM91,FuSh_PA92,Shi_JJIAM96}.
The “hot spots” conjecture has
been shown to hold on the Sierpinski gasket and higher dimensional
variants  
\cite{Ruan_SG,Ruan_SG3,LiRuan} but fail on the
hexagasket fractal \cite{Lau}. The hexagasket fractal is determined
by an iterated function system consisting of six
contractions. However, the analytic boundary in the sense of Kigami 
 studied in \cite{Lau} consists of  only three of the six fixed points of the
 iterated function system. On the other hand, the
 boundary of the Sierpinski gasket and its higher dimensional variants
 mentioned above consists of all of the fixed points of the iterated function
 system that determines the gasket. Therefore, one might wonder
 whether the failure of the ``hot spots'' conjecture on the hexagasket
 fractal might be due to the ``smaller'' boundary considered in \cite{Lau}. 
The Vicsek set \cite{Bar_LNM98,MalTep_MPAG03,Met_AAM93}  is
another type of fractal that has been studied heavily.  Zhou
\cite{Zhou} described the spectral decimation 
on a family of Vicsek sets, $VS_n$, in terms of Chebyshev polynomials,
and the authors of \cite{ConStrWhe} used Zhou's formulas to study the
Laplacian and spectral operators on the Vicsek set. We study in this
paper the Vicsek set $VS_2$ that 
is generated by five contractions; however, its analytic boundary
consists of only four of the five fixed points of the iterated
function systems.  Our
main theorem states that, unlike the hexagasket, the ``hot spots'' conjecture on
the Vicsek set is
true.  The proof of the main theorem is inspired by 
proofs in \cite{Ruan_SG,Ruan_SG3} and \cite{LiRuan}. It is,
however, more involved and technical.

The organization of the paper is as follows: the second part of the
introduction contains a background on the Vicsek set $VS_2$, and the energy
and the Laplacian on the Vicsek set. In section \ref{sec:2}, we review
 the Neumann Laplacian on $VS_2$, and show how to use the spectral
decimation to determine the second smallest eigenvalue of the Neumann
Laplacian together with a basis of its eigenspace. Section
\ref{sec:main} contains our main theorem. The proof of the theorem
follows relatively easily from Lemma \ref{lem:main} by an argument
similar to the one in \cite{Ruan_SG}. The proof of the lemma,
however, is very long and technical and occupies the entirety 
of Section \ref{sec:proof}. We placed the statement and proof of
some formulas used throughout the paper in the Appendix in order to
help with the readability of Section \ref{sec:proof}.

\textbf{Acknowledgment}: This paper grew up from the second's author project that  fulfilled the
requirements of 
the Honors program in the Department of Mathematics at the United
States Naval Academy under the supervision of the first author. We
thank all the members of the Department that 
read the Honors project  and gave us feedback, especially Professor
David Joyner, whose many suggestions led to a much better Honors
project and, we hope,  a better paper. The authors would also like to thank the
anonymous referee and language editor for many useful suggestions that led to an improvement in the
presentation of the paper.

\subsection{Background}
We begin by reviewing several  concepts in analysis on fractals as
applied to the Vicsek set.  First,  an iterative function system on a
complete metric space $X$ is a finite
set of contraction mappings  ${F_i}:X \rightarrow X$, $i=1,\dots, n$
\cite{Hut,Barns}. Given such an iterated function system
there exists a unique compact invariant set $K\subseteq X$; that is,
$K$ satisfies the following self-similar property, 
\[
K=F_1(K)\bigcup F_2(K)\bigcup \dots \bigcup F_n(K).
\]

The main object of study in this paper is the second order Vicsek set,
$VS_2$, which is the unique invariant subset of $\mathbb{R}^2$ of the
iterated function system defined by the following five similarities:  
\begin{equation}\label{eq:ifs}
F_i(x) = \frac{1}{3}(x-p_i) + p_i,
\end{equation}
where $p_1=(0,1)$, $p_2=(1,1)$, $p_3=(1,0)$, $p_4=(0,0)$, and $p_5=(1/2,1/2)$.



A picture of the Vicsek set is provided in the following figure.
\begin{figure}[h]
\begin{center}
\includegraphics[scale=0.6]{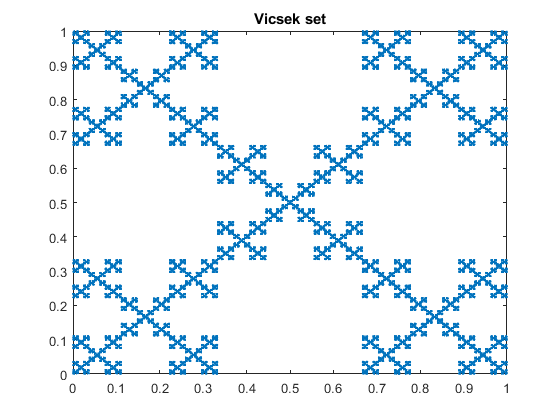}
\end{center}
\caption{Vicsek set.}
\end{figure}

The Vicsek set $VS_2$ is an example of what in the literature is called a p.c.f fractal
and, thus, Kigami's theory \cite{Kig} (see also \cite{Strichartz})
applies to $VS_2$. We describe next how this theory 
 is applied to the Vicsek set in order to define the standard
energy and Laplacian on $VS_2$.
A useful feature of the Vicsek is that it can be approximated by an
increasing sequence of graphs, $\Gamma_m$ as follows: the level zero
graph approximation of the Vicsek set, $\Gamma_0$, shown below,
consists of the set of vertices $V_0$ at $q_1 = (0,1)$, $q_2 = (1,1)$,
$q_3 = (1,0)$, and $q_4 = (0,0)$ connected as in a complete
graph. Note that $F_i(q_i) = q_i$, where $i={1,2,3,4}$. The set $V_0$ 
is also the boundary of the fractal in the sense of \cite{Kig,Zhou}.
We call $\Gamma_0$  the graph  0-cell.
\begin{figure}[h]
\begin{center}
\includegraphics[scale=0.6]{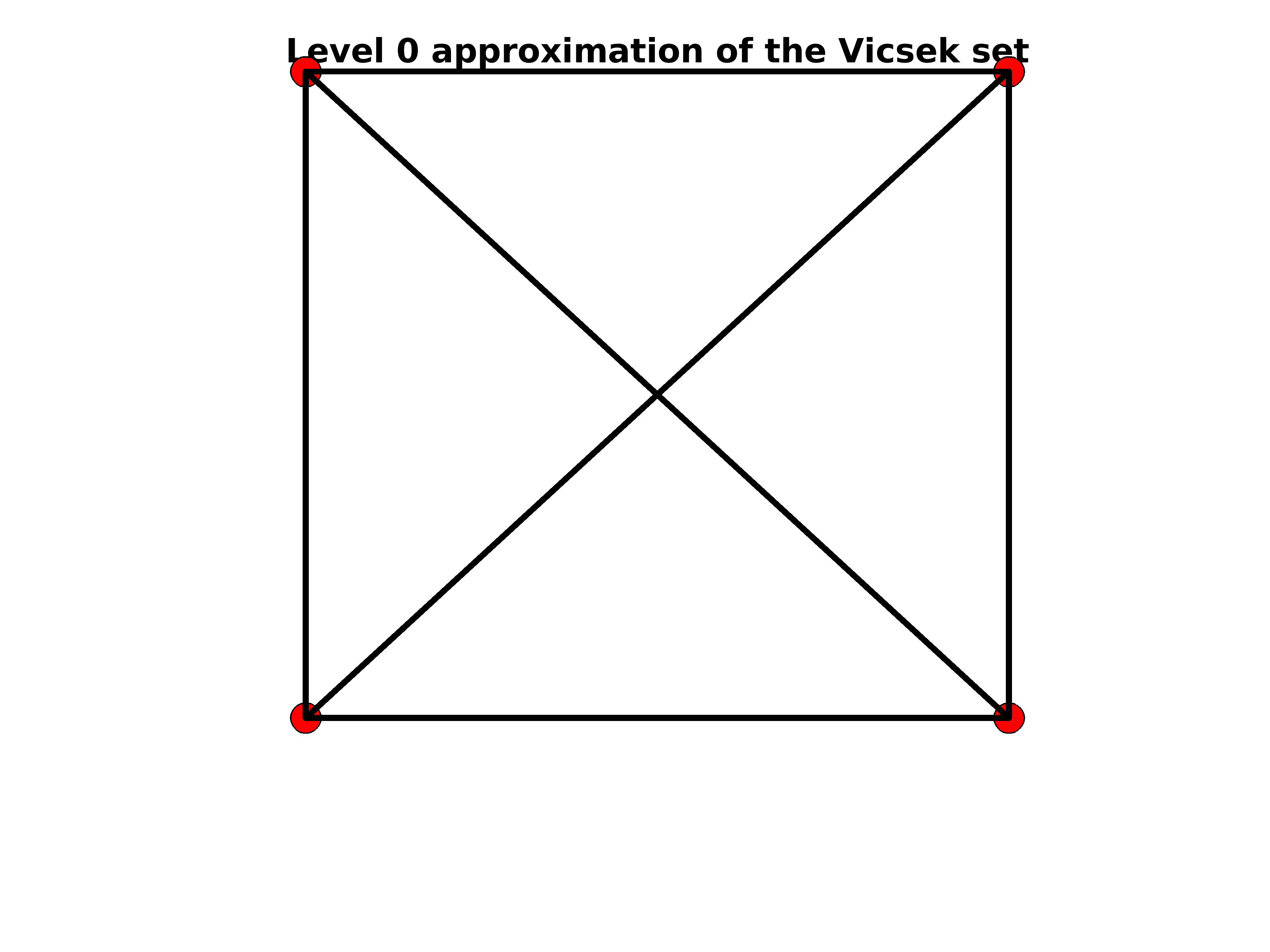}
\end{center}
\caption{Level 0 approximation.}
\end{figure}
The level 1 graph approximation of the Vicsek set, $\Gamma_1$ is
constructed as follows: each of the five scaled versions of the level
0 Vicsek set which comprise $\Gamma_1$ can be obtained by applying
$F_1(\Gamma_0)$, $F_2(\Gamma_0)$, etc. Therefore
\[
  \Gamma_1=F_1(\Gamma_0)\bigcup F_2(\Gamma_0)\bigcup
  F_3(\Gamma_0)\bigcup F_4(\Gamma_0) \bigcup
  F_5(\Gamma_0) 
\]
and, in particular, the set of
vertices of $\Gamma_1$ equals $V_1 = F_1(V_0) \bigcup F_2(V_0) \bigcup
\dots \bigcup F_5(V_0)$.
We call the sets $F_i(\Gamma_0)$ 
graph $1$-cells. We note that there are five graph 1-cells.
\begin{figure}[h]
\begin{center}
\includegraphics[scale=0.6]{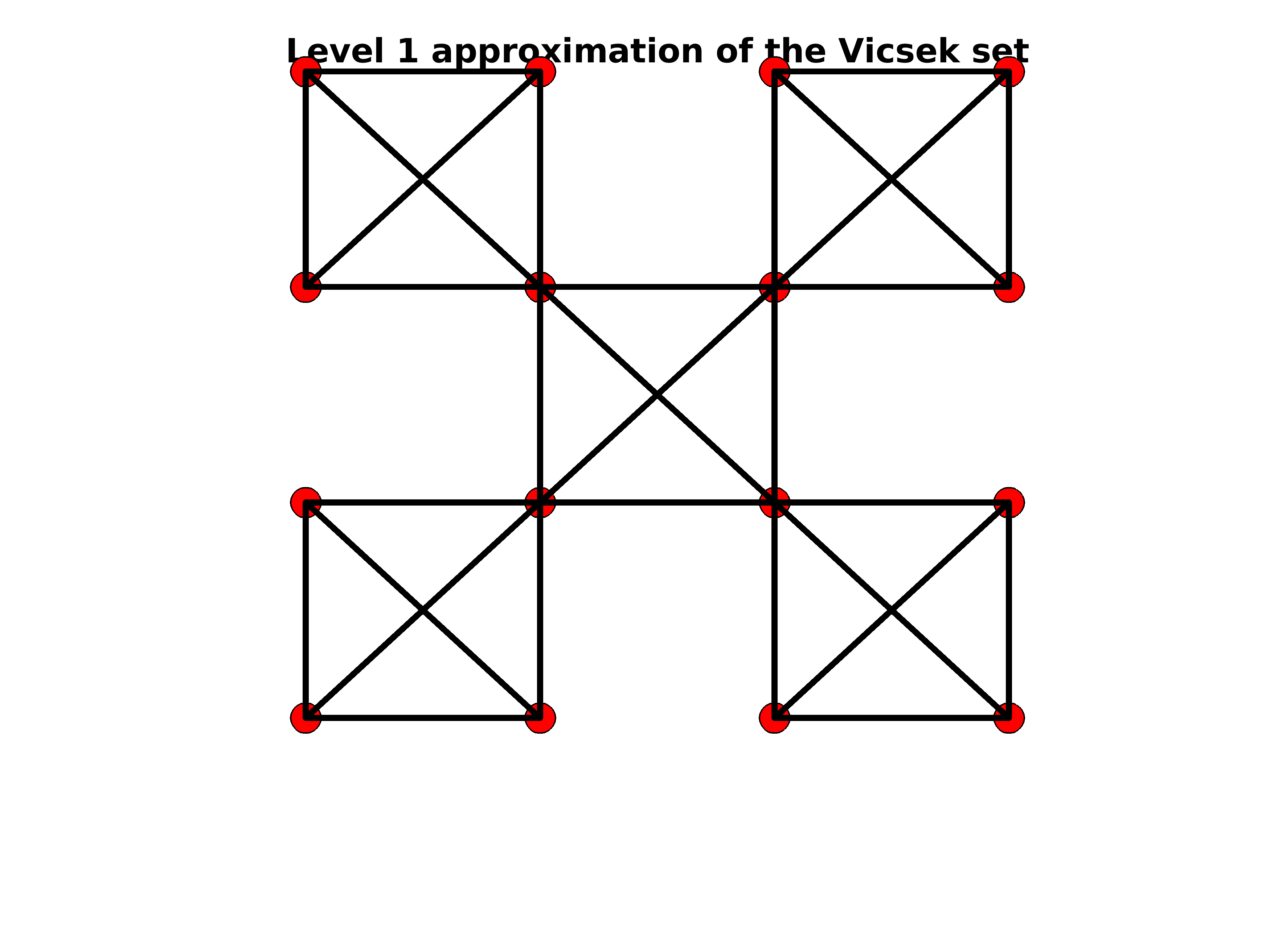}
\end{center}
\caption{Level 1 approximation.}
\end{figure}
In the general case, for $m\ge 1$, the level $m+1$ graph
approximation $\Gamma_{m+1}$ is obtained from the level $m$ graph
approximation $\Gamma_m$  via
\[
\Gamma_{m+1} = F_1(\Gamma_m) \bigcup F_2(\Gamma_m) \bigcup \dots \bigcup F_5(\Gamma_m).
\]
The image under the maps $F_i$ of the graph $m$-cells form the
graph $(m+1)$-cells. We note that each graph $m$-cell contains five graph $(m+1)$-cells.
That is, when going from level $m$ to level $m+1$, a graph $m$-cell is
going to be replaced by five graph $(m+1)$-cells. We write $V_m$ as the set of
vertices at the level $m$ graph approximation. Inductively, it can be
shown that $V_m \subset V_{m+1}$, and that $\bigcup_{m=0}^\infty V_m$ is a
dense subset of $VS_2$. Thus, it suffices to study continuous
functions on $\bigcup_{m=0}^\infty V_m$ and extend them via continuity
to $VS_2$. 

As detailed by \cite{Strichartz}, we can define graph energy at level
$m$ for the Vicsek set $VS_2$ as follows: 
\[
  E_m(u) = \sum_{\{(x,y)|x\sim y\}} |u(x)-u(y)|^2,
\]
where $x\sim y$
means that $x$ and $y$ are neighboring vertices in $\Gamma_m$ (that is, there
is an edge in $\Gamma_m$ joining $x$ and $y$). Energy
generally requires renormalization with a given renormalization
factor. For $VS_2$, the renormalization factor is equal to 3
\cite{Zhou}. Thus, for $VS_2$, the renormalized graph energy is
$\mathcal{E}_m (u) = 3^{-m}E_m(u)$. One can define the fractal's
energy, $\mathcal{E}(u)$, via \[ 
\mathcal{E}(u) = \lim_{m\rightarrow\infty} \mathcal{E}_m(u),
\]
with the domain of the energy, $\operatorname{dom}\mathcal{E}$,
consisting of all continuous functions $u$ on $VS_2$ such that
$\mathcal{E}(u)<\infty$. As detailed in \cite{Kig,Strichartz,Zhou},
the choice of the renormalization constant in the definition of
$\mathcal{E}_m$ guarantees that the domain of the energy is
non-trivial.  Now, knowing that the energy exists and can be written
as above, we know from \cite{Kig,Strichartz} that we can determine a
fractal's Laplacian. In dom$\mathcal{E}$, $\mathcal{E}$ extends via
the polarization formula to the bilinear function $\mathcal{E}(u,v)$. 
 We consider the
standard invariant measure $\mu$ on $VS_2$; that is, the unique
measure that satisfies the following property (see \cite{Hut} and
\cite{Barns} for details): 
\[
\int_{VS_2}f(x)\,d\mu(x)=\frac15\sum_{i=1}^5\int_{VS_2}f(F_i(x))\,d\mu(x).
\]
Define the weak formulation of the Laplacian as follows (\cite{Kig,Strichartz}): we say that a function $u\in\operatorname{dom}\mathcal{E}$  belongs to the domain $\operatorname{dom}\Delta$ of the Laplacian if there is a continuous function $f$ on $VS_2$ such that 
\[
\mathcal{E} (u,v) = - \int fv d\mu
\]
for all $ v\in \text{dom}_0 \mathcal{E}:=\{v\in
\text{dom}\mathcal{E}\,:\,v|_{V_0}=0\}$. In this case we write $\Delta
u=f$. As proven in \cite{Zhou}, there exists an equivalent pointwise
formula for the Laplacian on the Vicsek set. This formula
is given as the normalized limit of a series of graph approximations
and is written as follows for $VS_2$: 
\[
\Delta u(x) = \lim_{m\rightarrow\infty} 15^m \Delta_m u(x)
\]
for all $x$ not in the boundary $V_0$. Here the graph  Laplacian of
the graph $\Gamma_m$, denoted by $\Delta_m$, is defined by  
\begin{equation}\label{eq:Lap}
\Delta_m u(x) = \frac{1}{\text{deg} x} \sum_{y\sim x} (u(y) - u(x)) 
\end{equation}
for all $x\in V_m\setminus V_0$, where $\text{deg}x$ represents the number of
neighbors of $x$ in $\Gamma_m$. The degree of a vertex $x$ in $\Gamma_m$ is either $3$ or
$6$ for all $m\ge 1$.

\section{Neumann Laplacian and spectral decimation}
\label{sec:2}
We study the Neumann Laplacian in this paper. That is, the Laplacian
as defined with Neumann boundary conditions. Neumann boundary
conditions behave such that Equation \eqref{eq:Lap} holds for all
$x\in V_m$, including the boundary $V_0$. 

The Neumann Laplacian of the level 0 graph approximation $\Gamma_0$
for $VS_2$ is then given by the following matrix:
\[
\Delta_0=\left[
\begin{array}{cccc}
 1     &        -1/3  &         -1/3  &         -1/3     \\
      -1/3  &          1    &         -1/3    &       -1/3     \\
      -1/3   &        -1/3   &         1       &      -1/3     \\
      -1/3   &        -1/3    &       -1/3     &       1   
\end{array}
\right].
\]
The $1/3$ scaling factor outside of the matrix is derived from the
$1/\text{deg}x$ that appears in front of the summation in the equation for
$\Delta_m u(x)$. 

In order to study eigenvalues and eigenfunctions on $VS_2$, we use the
process of spectral decimation as described in \cite{Strichartz,Zhou}
that we review next. First, there is a local extension algorithm which
shows a unique way to extend a function $u$  that satisfies the
eigenvalue equation $-\Delta_mu=\lambda_mu$ on $V_m\setminus V_0$ to
a function that we still denote by $u$  that satisfies the eigenvalue
equation $-\Delta_{m+1}u=\lambda_{m+1}u$ on $V_{m+1} \setminus V_0$. Moreover there
exists a rational function $R(\lambda)$ such that 
$\lambda_m$ = $R(\lambda_{m+1})$ if $\lambda_m$ is not a forbidden
eigenvalue. That is, a singularity of the function $R$. It is
"forbidden" to decimate to such eigenvalues. Because forbidden
eigenvalues do not have a predecessor, i.e. there is no
$\lambda_{m-1}$ corresponding to $\lambda_m$, we say that forbidden
eigenvalues are "born" at a level of approximation $m$.
Spectral decimation for $VS_2$ is performed as follows.  
First, define 
\begin{align}\label{eq:f2}
f_2(\lambda) &= T_2(3\lambda - 1) - 3T_1(3\lambda - 1)\\
g_2(\lambda) &= U_1(3\lambda - 1) - U_0(3\lambda - 1)\nonumber\\
h_2(\lambda) &= U_1(3\lambda - 1) - 3U_0(3\lambda - 1)\nonumber
\end{align}
where $T_2$ and $U_2$ represent the Chebyshev polynomials of the 1st
and 2nd kind, i.e.  $T_1(\lambda)=\lambda$,
$T_2(\lambda)=2\lambda^2-1$, $U_0(\lambda)=1$, and
$U_1(x)=2\lambda$. Therefore, $f_2(\lambda)=18\lambda^2-21\lambda+4$,
$g_2(\lambda)=6\lambda-3$, and $h_2(\lambda)=6\lambda-5$. Zhou proved
in \cite{Zhou} (see also \cite{Strichartz,ConStrWhe})
that the spectral decimation function $R$ is
\begin{equation}\label{eq:R}
R(\lambda)=\lambda g_2(\lambda)h_2(\lambda)=36\lambda^3-48\lambda^2+15\lambda.
\end{equation}
Additionally, the forbidden eigenvalues of $VS_2$ are 4/3 and the
zeroes of $f_2$ and $g_2$, 
0, 1/2, and $(7\pm
17^{0.5})/12$. 
The extension of eigenfunctions of $VS_2$ from one  level to the next
is given by \cite{ConStrWhe}: 
\begin{equation}\label{eq:matrix}
-(X+\lambda M)^{-1}J=\gamma\left[
\begin{array}{cccccccccccc}
 a & b & a & c & c & d & d & c & c & c & d & c\\
 c & d & c & a & a & b & d & c & c & c & d & c\\
 c & d & c & c & c & d & b & a & a & c & d & c\\
 c & d & c & c & c & d & d & c & c & a & b & a   
\end{array}
\right]^T
\end{equation}
where 
\begin{align*}
a &= 9-42\lambda +36\lambda^2,   c = 1\\
b &= 6(1-4\lambda+3\lambda^2),   d = 2-3\lambda\\
\gamma &= \frac{1}{3(4-29\lambda+60\lambda^2-36\lambda^3)}.
\end{align*}
Note that $a,b,c,d$ and $\gamma$ are functions of $\lambda$. Hereafter
$\lambda$ is any number that is not a forbidden eigenvalue of 
$VS_2$. In $\Gamma_1$, $J$ is equal to the $V_0 \times (V_1 \setminus
V_0$) adjacency matrix, $X$ is the adjacency matrix of ($V_1\setminus
V_0$) with the degrees of every vertex as its diagonal entries, and
$M$ is a diagonal matrix such that $M_{ii}$ = -$X_{ii}$. Going from
$\Gamma_m$ to $\Gamma_{m+1}$, this matrix is applied to the vertices
of one graph $m$-cell to obtain the values on the five new graph
$(m+1)$-subcells.

We follow the convention from \cite[Figure 4]{ConStrWhe} when it comes
to the labeling of the columns of the matrix \eqref{eq:matrix}. That
is, the first column corresponds to the point $q_5=F_1(q_2)$, the
second column corresponds to the point $q_6=F_1(q_3)=F_5(q_1)$, the
third column corresponds to $q_7=F_1(q_4)$, the fourth column
corresponds to $q_8=F_2(q_1)$ and so on. To better understand the meaning of the matrix
\eqref{eq:matrix}, let $\lambda_0$ be an eigenvalue of $\Delta_0$, and let
$u$ be a $\lambda_0$ eigenvector on $V_0$. To simplify the notation,
write $u(q_i)=u_i$, $i=1,2,3,4$.
 Let $\lambda_1$ be a
solution of $R(\lambda)=\lambda_0$. Then $u$ is extended to a
$\lambda_1$ eigenvector on $V_1$ using \eqref{eq:matrix} as described next. 
First, $u|_{V_0}$ does not change. The
extension of $u$ to $q_5=F_1(q_2)$ is computed using the first column
of \eqref{eq:matrix} as follows:
\[
  u(q_5)=\gamma au_1+\gamma cu_2+\gamma
  cu_3+\gamma cu_4,
\]
where $\gamma, a$, and $c$ are evaluated at $\lambda_1$. The extension
of $u$ to $q_6$ is computed using the second column of \eqref{eq:matrix} as follows:
\[
  u(q_6)=\gamma bu_1+\gamma du_2+\gamma du_3+\gamma du_4,
\]
where $\gamma, b$, and $d$ are evaluated at $\lambda_1$. The
computation of the extension of $u$ to the remaining vertices in $V_1$
is computed based on the corresponding columns in \eqref{eq:matrix}.



The eigenvalue extension function $R(\lambda)=36\lambda^3 -
48\lambda^2 + 15\lambda$ has three local inverses \cite{ConStrWhe}. Let $\phi_1$,
$\phi_2$, $\phi_3$ denote the inverse functions 
of $R$ in increasing order; that is, $\phi_1$ is the inverse of
$R(\lambda)$ on the interval $(0,\frac{8-19^{0.5}}{18})$, $\phi_2$ is
the inverse of $R(\lambda)$ on
$(\frac{8-19^{0.5}}{18},\frac{8+19^{0.5}}{18})$, and $\phi_3$ is the
inverse of $R(\lambda)$ on $(\frac{8+19^{0.5}}{18},1)$. Observe $\phi_1(x)<
\phi_2(y)<\phi_3(z)$ for all $x,y,z$ in the corresponding domain, and
$\phi_2$ is decreasing while $\phi_1$ and $\phi_3$ are increasing. Let
$\rho_2 = 15$ be the renormalization factor for the
Laplacian. 
The Neumann eigenvalues are non-negative and they accumulate at
$\infty$. Then the rules for spectral
decimation in the case of $VS_2$ are summarized as follows
\cite{ConStrWhe}: 
\begin{enumerate}
\item For each Neumann eigenvalue $\lambda$, there is an infinite word
  $\{\omega_j\}_{j=1}^\infty$, where $\omega_j\in \{1,2,3\}$ for all
  $j\ge 1$, such that  $\lambda$ equals  
\begin{align*}
\lim_{m\rightarrow\infty} 15^m\phi_{\omega_m} \circ
  \phi_{\omega_{m-1}} \circ \ ... \circ \phi_{\omega_1}(0) \text{ or}
  \\ 
\lim_{m\rightarrow\infty} 15^{m+k}\phi_{\omega_m} \circ \phi_{\omega_{m-1}} \circ \ ... \circ \phi_{\omega_1}(4/3).
\end{align*}
 The existence of the limit is proven in \cite{Zhou}. In the first
 case, the eigenvalue is in the 0-series, and in the second series
 the eigenvalue is in the 4/3 series born on level  $k$
\item All but a finite number of the $\omega_m$ are equal to 1
\item For the 0-series, the first $\omega_j$ with  $\omega_j \ne
  1$ must be an odd number, and for the 4/3 series, $\omega_1$ must be
  an odd number but $\omega_1 \ne 3$ 
\item The multiplicity of each eigenvalue in the 0-series is 1, while the multiplicity of each 4/3-series eigenvalue on level $k$ is $2(5)^k + 1$.
\end{enumerate}

In the remainder of the paper, we use the spectral decimation as
applied to the following setup. Let $\lambda_0 = 4/3$ be the second
smallest eigenvalue of $\Delta_0$. Then $\lambda_0 = 4/3$ has
multiplicity 3 and a basis  for its eigenspace on $V_0$ is given by
the following eigenvectors: 
\begin{equation}\label{eq:bases}
u_1=\left[
\begin{array}{r}
 1\\
 0\\
 0\\
 -1 
\end{array}
\right],\,
u_2=\left[
\begin{array}{r}
 0\\
 1\\
 0\\
 -1 
\end{array}
\right],\,
u_3=\left[
\begin{array}{r}
 0\\
 0\\
 1\\
 -1 
\end{array}
\right].
\end{equation}
Extend $\lambda_0$ at all levels along
$\phi_1$. That is, define  
\begin{equation}\label{eq:lambdam}
\lambda_m=\phi_1(\lambda_{m-1})\,\text{ for all }\,m\ge 1
\end{equation}
and define 
\begin{equation}\label{eq:lambda}
\lambda^{(2)} = \lim_{m\rightarrow\infty} 15^m \lambda_m.
\end{equation}
We extend $u_1,u_2$ and $u_3$ to eigenvectors of $\lambda^{(2)}$ on
the Vicsek set via the spectral decimation. 
Figure
\ref{fig:extu1} presents the extension of $u_1$ from $V_0$ to 
$V_2$ using the spectral decimation described above. 
\begin{figure}[h]
  \centering
  \includegraphics[scale=0.6]{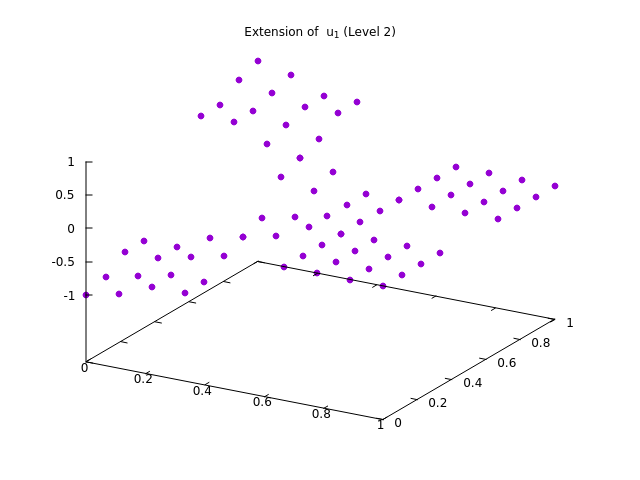}
  \caption{Extension of $u_1$ to $V_2$.}
  \label{fig:extu1}
\end{figure}
The
following important fact  follows from \cite[Theorem 2.2]{ConStrWhe}.

\begin{prop}
 $\lambda^{(2)}$ is the 2nd smallest eigenvalue of the Neumann Laplacian on the Vicsek set.
\end{prop}


An important fact later in the paper is that $\lambda_1 =\phi_1(4/3)=
1/6$. One can easily check that $R(1/6)=4/3$. The fact that $1/6$ is
the smallest preimage of $4/3$ under $R(\lambda)$ follows from
numerically solving the equation $R(\lambda)=4/3$. For future
reference, using Maxima
approximation, we find  that $\lambda^{(2)}$ is roughly 
\[
2.601813889315113780749839. 
\]



 
 
 
 
 
 
 
 
 
 
 
 
\section{The "Hot Spots" conjecture for the Vicsek Set}
\label{sec:main}
We are now able to state the main theorem. The proof of the theorem
follows from the  main lemma
\ref{lem:main}.  The proof of this lemma is long and technical,  and  it will occupy the
entirety of the next section. 
Our main result states that every
eigenfunction of the second smallest eigenvalue of the Neumann
Laplacian on the Vicsek set attains its maximum and minimum on the
boundary. Our approach is based on techniques from \cite{Ruan_SG} and
\cite{Ruan_SG3}. Our computations, however, are much more involved due
to the  complexity of    the spectral decimation matrix \eqref{eq:matrix}.  

\begin{thm}[Main Theorem]\label{thm:main}
Let $VS_2$ be the Vicsek set and $\Delta$ be the Neumann Laplacian as
described in the previous section. Then every eigenfunction of the
second smallest eigenvalue $\lambda^{(2)}$ of $\Delta$ attains its minimum and maximum
value on the boundary $V_0$. 
\end{thm}

In order to prove this theorem, we begin by recalling the space
of finite words that the five iterated functions \eqref{eq:ifs} of the
Vicsek set generate. Let $\Sigma =\{1,2,3,4,5\}$ be the corresponding
alphabet and let $\Sigma^m = \{\omega_1\dots \omega_j\dots
\omega_m\,|\, j\in \{1,2,3,4,5\}\,  \}$ be the set of words of length
$m$. Note that in order to simplify the notation we do not separate
the letters in a word by commas. So, for example, we write $14$
for the word of length 2 that is formed with the letters 1 and
4. Define $\Sigma^*$ = $\bigcup_{m=0}^{\infty} \Sigma^m$ as the set 
of all finite words. For every $\omega\in\Sigma^m$, we write
$\vert\omega\vert := m$ for the length of $\omega$. Let $\emptyset$
denote the empty word and |$\emptyset$| = 0. Furthermore, $\Sigma^0$ =
$\{\emptyset\}$. If $\omega,\nu\in \Sigma^*$ then we write $\omega\nu$
for the word obtained by concatenating the two words $\omega$ and
$\nu$ together. In particular, we will create words by adding just one
letter. For example, $\omega i$ is the word that adds the letter $i$ at the end of
$\omega$. If all the letters of $\omega$ are the same, say equal to
$i\in\{1,2,3,4,5\}$, then we write $\omega$ as $[i]^m$, where $m$ is the length of
$\omega$. 

Recall that $V_0 = \{q_1, q_2, q_3, q_4\}$ where $F_i(q_i) = q_i$, $i=1,2,3,4$. We let
$\overline{\Sigma^*}=\Sigma^*\times \{1,2,3,4\}$. For $\omega$ $\in$ 
$\Sigma^m$, we write $F_{\omega} = F_{\omega_1}\circ F_{\omega_2}
\circ ... \circ F_{\omega_m}$ and for $(\omega,i)\in
\overline{\Sigma^*}$ we write $q_{\omega,i} = F_{\omega}(q_i)$. Notice
that $\{q_{\omega,i}\}_{(\omega,i)\in
  \overline{\Sigma^*}}=\bigcup_{m=0}^{\infty} V_m$ forms a dense subset
of $VS_2$.  

Recall from \eqref{eq:lambda} that
$\lambda^{(2)}=\lim_{m\to\infty}\lambda_m$ is the second smallest
eigenvalue of the Neumann Laplacian where $\lambda_m$ are defined in
\eqref{eq:lambdam}. Recall also that $\lambda_0=4/3$, $\lambda_1=1/6$, and
they are related   via $\lambda_m=R(\lambda_{m+1})$ for all $m\ge 0$,
where $R(\lambda)$ is the eigenvalue extension function
\eqref{eq:R}. We let $EF_2$ be the the eigenspace of 
$\lambda^{(2)}$. Then $EF_2$ is a three-dimensional vector space by
the spectral decimation and we pick the bases given by $u_1$, $u_2$,
$u_3$ $\in$ $EF_2$ with $u_1(q_1) = 1$, $u_1(q_2) = 0$, $u_1(q_3) =
0$, $u_1(q_4) = -1$, $u_2(q_1) = 0$, $u_2(q_2) = 1$, $u_2(q_3) = 0$,
$u_2(q_4) = -1$, $u_3(q_1) = 0$, $u_3(q_2) = 0$, $u_3(q_3) = 1$,
$u_3(q_4) = -1$. That is, we pick $u_1,\,u_2,\,u_3$ to be the
extensions via the spectral decimation of the $4/3$-eigenvectors on
$V_0$ defined in \eqref{eq:bases}. 

Following some ideas from \cite{Ruan_SG}, we use the bases that we
picked to define a partition of unity on
$\overline{\Sigma^*}$. Specifically, we define the functions $f, g,
h,$ and $k$ with domain $\overline{\Sigma^*}$ via:
\begin{align}
f(\omega,i) &= \frac34u_1(q_{\omega,i})-\frac14u_2(q_{\omega,i})-\frac14u_3(q_{\omega,i})+\frac14\label{eq:f}\\
g(\omega,i) &= -\frac14u_1(q_{\omega,i})+\frac34u_2(q_{\omega,i})-\frac14u_3(q_{\omega,i})+\frac14\label{eq:g}\\
h(\omega,i) &= -\frac14u_1(q_{\omega,i})-\frac14u_2(q_{\omega,i})+\frac34u_3(q_{\omega,i})+\frac14\label{eq:h}\\
k(\omega,i) &= -\frac14u_1(q_{\omega,i})-\frac14u_2(q_{\omega,i})-\frac14u_3(q_{\omega,i})+\frac14\label{eq:k}
\end{align}
Then they satisfy the following crucial lemma.
\begin{lem}\label{lem:main}

We have $0\le f(\omega,i), g(\omega,i),h(\omega,i),k(\omega,i) \le 1$
and $f(\omega,i)+g(\omega,i)+h(\omega,i)+k(\omega,i) = 1$ for every
$\omega \in  \Sigma^*$ and for every $i \in
\{1,2,3,4\}$. Additionally, $f(\emptyset,i)=\delta_{1i}$,
$g(\emptyset,i)=\delta_{2i}$, $h(\emptyset,i)=\delta_{3i}$,
$k(\emptyset,i)=\delta_{4i}$ where $\delta_{ij}$ is the Kronecker-Delta
function. 
\end{lem}

As mentioned at the beginning of the section, the proof of the lemma
is long, technical, and will occupy the next section of this
paper. From the lemma, the theorem is proven as follows. 

\begin{proof}[Proof of Theorem \ref{thm:main}]
  Notice that the functions $f, g, h,$ and $k$ are related as follows:
  \begin{eqnarray*}
    f(\omega,i)-k(\omega,i)&=&u_1(q_{\omega,i}),\\
    g(\omega,i)-k(\omega,i)&=&u_2(q_{\omega,i}),\\
    h(\omega,i)-k(\omega,i)&=&u_3(q_{\omega,i}).
  \end{eqnarray*}
 Let $u \in EF_2$. Then
there exist constants $c_1, c_2, c_3$ such that $u(x) =
c_1u_1(x)+c_2u_2(x)+c_3u_3(x)$ for all $x \in VS_2$ because
$u_1,u_2,u_3$ form a basis for $EF_2$. It follows that 
\begin{equation*}
u(q_{\omega,i}) = c_1f(\omega,i)+c_2g(q\omega,i)+c_3h(\omega,i)+(-c_1-c_2-c_3)k(\omega,i).
\end{equation*}
Lemma $\ref{lem:main}$ implies that the maximum/minimum value of
$u$ on $\bigcup_{m=0}^{\infty} V_m$ is given by the maximum/minimum of
the values $c_1, c_2, c_3$ and $(-c_1-c_2-c_3)$, which are the values
of $u$ on the boundary $V_0$. Since $\bigcup_{m=0}^{\infty} V_m$ is
dense in $VS_2$ and $u$ is continuous the theorem follows. 
\end{proof}

\section{Proof of Lemma \ref{lem:main}}
\label{sec:proof}

The proof of Lemma \ref{lem:main} will occupy the rest of this
paper. 
We provide
first a short summary of the proof in order to help the reader
navigate  through the many lemmas that follow. First, an easy proof
shows that $f(\emptyset,i)=\delta_{1i}$ (Lemma \ref{lem:deltai}). We observe that it
suffices to prove the statement for $f$  for the words $\omega$
that begin only with the letters $1,2$, and $5$, because $f$ restricted to
words that begin with $3$ and $4$ equals a ``rotation'' of $f$
restricted to words that begin with $2$ (see Remark
\ref{rem:rot}). The crucial Lemma \ref{lem:fw} proves recursive formulas  for
$f(\omega,j)$ when one increases the length of $\omega$ by 1. Using
these formulas we are able to compute explicitly $f(\omega,j)$ for words of
the form $\omega=[1]^m$ (Lemma \ref{lem:1m}), $\omega=[2]^m$ (Lemma
\ref{lem:2m}), and $\omega=[5]^m$ (Lemma \ref{lem:5m}). Recall that we write
$[i]^m$ for the word of length $m$ consisting only on the letter $i$, i.e.
$[i]^m=ii\dots i$ ($m$-times). Using these explicit formulas, we prove
in Proposition \ref{prop:maxf} that the minimum of $f$ is 0
and the maximum of $f$ is 1. We finish the proof of 
Lemma \ref{lem:main} by describing how to recover the same results for
$g,h$, and $k$ from the results proved for $f$.
To improve the readability of this section, we leave the statements
and proofs of some useful formulas until the Appendix (see Lemma \ref{lem:formulas}).

We continue to use the notation described in the previous section. In
particular, $f,g,h$ and $k$ are the functions defined  in \eqref{eq:f}, \eqref{eq:g},
  \eqref{eq:h}, and \eqref{eq:k}. Recall  that we do not separate
  the letters in a word $\omega\in \Sigma^*$.  Therefore,
  $123$ is the word of length 3 that consists of the letters 1, 2 and
  3 (we read this word  ``one, two, three'', as opposed to ``one hundred
  twenty three''). This word will correspond to the
  composition $F_1\circ F_2\circ F_3$.

  We
break the proof of the main Lemma \ref{lem:main} into a series of 
lemmas. We begin with the easier part.

\begin{lem}\label{lem:deltai}
  $f(\emptyset,i)=\delta_{1i}$,
$g(\emptyset,i)=\delta_{2i}$, $h(\emptyset,i)=\delta_{3i}$,
$k(\emptyset,i)=\delta_{4i}$, where $\emptyset$ is the empty word and
$\delta_{ij}$ is the Kronecker-Delta function. 
\end{lem}
\begin{proof}
From the definition of the functions, it is clear that
$f(\omega,i)+g(\omega,i)+h(\omega,i)+k(\omega,i) = 1$.  
Notice that
\begin{align*}
f(\emptyset,1) &= \frac34u_1(q_{1})-\frac14u_2(q_{1})-\frac14u_3(q_{1})+\frac14 = \frac34+\frac14 = 1\\
f(\emptyset,2) &= \frac34u_1(q_{2})-\frac14u_2(q_{2})-\frac14u_3(q_{2})+\frac14 = \frac14-\frac14 = 0\\
f(\emptyset,3) &= \frac34u_1(q_{3})-\frac14u_2(q_{3})-\frac14u_3(q_{3})+\frac14 = \frac14-\frac14 = 0\\
f(\emptyset,4) &= \frac34u_1(q_{4})-\frac14u_2(q_{4})-\frac14u_3(q_{4})+\frac14 = \frac14-\frac14 = 0.
\end{align*}
Hence, $f(\emptyset,i) = \delta_{1,i}$. Similar statements hold for
$g$ and $h$ by symmetry, and similar computation shows that the result
also holds for $k$.
\end{proof}

The hard part that remains is to prove that each of
these functions is between 0 and 1. 
First, we begin by proving that $f(\omega i,i) = f(\omega,i)$ and
describe recursive relations satisfied by $f$.  
\begin{eqnarray*}
f(\omega i,i) &=& \frac34 u_1(F_{\omega i}(q_i)) - \frac14 u_2(F_{\omega i}(q_i)) - \frac14 u_3(F_{\omega i}(q_i)) + \frac14\\
&=& \frac34 u_1(F_{\omega}(F_i(q_i))) - \frac14 u_2(F_{\omega}(F_i(q_i))) - \frac14 u_3(F_{\omega}(F_i(q_i))) + \frac14\\
&=& \frac34 u_1(F_{\omega}(q_i)) - \frac14 u_2(F_{\omega}(q_i)) -
    \frac14 u_3(F_{\omega}(q_i)) + \frac14\\
  &=& f(\omega, i)
\end{eqnarray*}
because $F(q_i)=q_i$. So $f(\omega i, i)$ = $f(\omega, i)$ for all $i\in\{1,2,3,4\}$.

For the rest of the paper we write $a_m$, $b_m$, $c_m$, $d_m$ for the elements of the extension
matrix  \eqref{eq:matrix} evaluated at $\lambda_m$ for all  $m\ge 0$
and 
\begin{eqnarray}
\alpha_m&:=&\gamma_ma_m=\frac{9-42\lambda_m+36\lambda_m^2}{3(4-29\lambda_m+60\lambda_m^2-36\lambda_m^3)}\label{eq:alpham}\\
\beta_m&:=&\gamma_mb_m=\frac{6(1-4\lambda_m+3\lambda_m^2)}{3(4-29\lambda_m+60\lambda_m^2-36\lambda_m^3)}\label{eq:betam}\\
\chi_m&:=&\gamma_mc_m=\frac{1}{3(4-29\lambda_m+60\lambda_m^2-36\lambda_m^3)}\label{eq:chim}\\
\delta_m&:=&\gamma_md_m=\frac{2-3\lambda_m}{3(4-29\lambda_m+60\lambda_m^2-36\lambda_m^3)}\label{eq:deltam}.
\end{eqnarray}

\begin{lem}\label{lem:fw}
  The following formulas hold for all $\omega\in \Sigma^*$\verb!:!
\begin{align}
f(\omega 1,2) = f(\omega 1,4) &=\alpha_{m+1}f(\omega,1) + \chi_{m+1}f(\omega,2) + \chi_{m+1}f(\omega,3) + \chi_{m+1}f(\omega,4)\label{eq:fw12}\\
&+ \frac14(1-\alpha_{m+1} -3\chi_{m+1}),\nonumber\\
f(\omega 2,1) = f(\omega 2,3) &= \chi_{m+1}f(\omega,1) + \alpha_{m+1}f(\omega,2) + \chi_{m+1}f(\omega,3) + \chi_{m+1}f(\omega,4)\label{eq:fw21}\\
&+ \frac14(1-\alpha_{m+1} -3\chi_{m+1}),\nonumber\\
f(\omega 3,2) = f(\omega 3,4) &= \chi_{m+1}f(\omega,1) + \chi_{m+1}f(\omega,2) + \alpha_{m+1}f(\omega,3) + \chi_{m+1}f(\omega,4)\label{eq:fw32}\\
&+ \frac14(1-\alpha_{m+1} -3\chi_{m+1}).\nonumber\\
f(\omega4,1)=f(\omega4,3)&= \chi_{m+1}f(\omega,1) + \chi_{m+1}f(\omega,2) + \chi_{m+1}f(\omega,3) + \alpha_{m+1}f(\omega,4)\label{eq:fw43}\\
&+ \frac14(1-\alpha_{m+1} -3\chi_{m+1}) \nonumber.\\
f(\omega5,1)=f(\omega 1,3) &= \beta_{m+1}f(\omega,1) + \delta_{m+1}f(\omega,2) + \delta_{m+1}f(\omega,3) + \delta_{m+1}f(\omega,4)\label{eq:fw13}\\
&+ \frac14(1-\beta_{m+1} -3\delta_{m+1}),\nonumber\\
f(\omega5,2)=f(\omega 2,4) &= \delta_{m+1}f(\omega,1) + \beta_{m+1}f(\omega,2) + \delta_{m+1}f(\omega,3) + \delta_{m+1}f(\omega,4)\label{eq:fw24}\\
&+ \frac14(1-\beta_{m+1} -3\delta_{m+1}),\nonumber\\
f(\omega5,3)=f(\omega 3,1) &= \delta_{m+1}f(\omega,1) + \delta_{m+1}f(\omega,2) + \beta_{m+1}f(\omega,3)+ \delta_{m+1}f(\omega,4)\label{eq:fw31}\\
& + \frac14(1-\beta_{m+1} -3\delta_{m+1}),\nonumber\\
f(\omega5,4)=f(\omega 4,2) &= \delta_{m+1}f(\omega,1) + \delta_{m+1}f(\omega,2) + \delta_{m+1}f(\omega,3) + \beta_{m+1}f(\omega,4)\label{eq:fw42}\\
&+ \frac14(1-\beta_{m+1} -3\delta_{m+1}),\nonumber
\end{align}
\end{lem}
\begin{proof}
  Let $\omega\in \Sigma^m$ for some $m\ge 0$. That is, $\vert \omega\vert=m$.
We prove  the case of $f(\omega 4,3)$ as follows:
\begin{eqnarray*}
f(\omega4,3) &=& \frac34 u_1(F_{\omega 4}(q_3)) - \frac14 u_2(F_{\omega 1}(q_3))- \frac14 u_3(F_{\omega 4}(q_3)) + \frac14
\\
&=& \frac34 [\chi_{m+1}u_1(F_{\omega}(q_1)) + \chi_{m+1}u_1(F_{\omega}(q_2)) + \chi_{m+1}u_1(F_{\omega}(q_3)) + \alpha_{m+1}u_1(F_{\omega}(q_4)) ]
\\
&&- \frac14 [\chi_{m+1}u_2(F_{\omega}(q_1)) + \chi_{m+1}u_2(F_{\omega}(q_2)) + \chi_{m+1}u_2(F_{\omega}(q_3)) + \alpha_{m+1}u_2(F_{\omega}(q_4)) ]
\\
&&-\frac14 [\chi_{m+1}u_3(F_{\omega}(q_1)) + \chi_{m+1}u_3(F_{\omega}(q_2)) + \chi_{m+1}u_3(F_{\omega}(q_3)) + \alpha_{m+1}u_3(F_{\omega}(q_4)) ]
\\
&&+ \frac14.
\end{eqnarray*}
Therefore,
\[
f(\omega4,3)= \chi_{m+1}f(\omega,1) + \chi_{m+1}f(\omega,2) +
\chi_{m+1}f(\omega,3) + \alpha_{m+1}f(\omega,4) +
\frac14(1-\alpha_{m+1} -3\chi_{m+1}).  
\]
The fact that $f(\omega4,1)=f(\omega4,3)$ follows from the fact that
the corresponding columns in the extension matrix \eqref{eq:matrix}
are identical. Therefore \eqref{eq:fw43} holds.
The remaining formulas follow by similar computation. The fact that
$f(\omega5,1)=f(\omega1,3)$, $f(\omega5,2)=f(\omega2,4)$,
$f(\omega5,3)=f(\omega3,1)$, and $f(\omega5,4)=f(\omega4,2)$ follows from  the fact that
$F_5(q_1)=F_1(q_3)$, $F_5(q_2)=F_2(q_4)$, $F_5(q_3)=F_3(q_1)$ and $F_5(q_4)=F_4(q_2)$.
\end{proof}

As an immediate consequence of Lemma \ref{lem:fw}, we obtain the following.
\begin{lem}\label{lem:w234}
  Let $\omega\in \Sigma^*$.
  \begin{enumerate}
  \item If $f(\omega,1)=f(\omega,2)=f(\omega,3)$ then
    \[
      f(\omega1,2)=f(\omega1,4)=f(\omega2,1)=f(\omega2,3)=f(\omega3,2)=f(\omega3,4)
    \]
    and $f(\omega1,3)=f(\omega2,4)=f(\omega3,1)$.
  \item If $f(\omega,1)=f(\omega,2)=f(\omega,4)$ then
    \[
      f(\omega1,2)=f(\omega1,4)=f(\omega2,1)=f(\omega2,3)=f(\omega4,1)=f(\omega4,3)
    \]
    and $f(\omega1,3)=f(\omega2,4)=f(\omega4,2)$.
  \item  If $f(\omega,2)=f(\omega,3)=f(\omega,4)$,
    then
    \[
      f(\omega2,1)=f(\omega2,3)=f(\omega3,2)=f(\omega3,4)=f(\omega4,1)=f(\omega4,3)
    \]
    and $f(\omega2,4)=f(\omega3,1)=f(\omega4,2)$.
  \end{enumerate}
\end{lem}

\begin{rem}\label{rem:rot}
 Using Lemma \ref{lem:fw} and Lemma \ref{lem:w234}, one can prove
 inductively the following symmetries (``rotations'') of the function
 $f$:
 \begin{enumerate}
 \item Let $R_1:\Sigma\to\Sigma$ be the permutation that flips 2 and
   4, $R_1=(2,4)$. We denote
   by $R_1$ its extension to $\Sigma^*$ as well. Then
   $f(\omega,i)=f(R_1(\omega),R_1(i))$ for all
   $(\omega,i)\in\overline{\Sigma^*}$.
 \item Let $R_2:\Sigma\to\Sigma$ be the permutation defined by $R_2=(1,2,3,4)$.
 Then
   $f(2\omega,i)=f(3R_2(\omega),R_2(i))$ for all $\omega\in \Sigma^*$
   and $i\in\{1,2,3,4\}$.
 \end{enumerate}
 We call $R_1$ and $R_2$ ``rotations'' since, if we view $f$ defined
 on $VS_2$ via the  projection  $\pi:\overline{\Sigma^*}\to VS_2$
 defined by $\pi(\omega,i):=q_{\omega,i}$, then $R_1$ flips
 $\cup_{m\ge 0}V_m$ around
 the diagonal going from $q_1$ to $q_3$, and $R_2$ rotates the  2-cell
 $F_2(\cup_{m\ge0}V_m)$ by $90^\circ$ and moves it into the  2-cell
 $F_3(\cup_{m\ge0}V_m)$.

 There is a permutation that rotates $F_2(\cup_{m\ge0}V_m)$ and moves
 it into the  2-cell $F_4(\cup_{m\ge0}V_m)$.
 As a consequence, in the following we will only consider words
 $\omega\in \Sigma^*$ whose first letter is either $1,2$ or $5$.
\end{rem}

 We begin by proving that 0 $\le f([1]^m,j), f([2]^m,j),
f([5]^m,j) \le 1$ for all $j\in \{1,2,3,4 \}$. 
We have already shown that $f([1]^m,1) = 1$ for all $m\ge0$, so we now
prove that 0 $\le f([1]^m,j)\le 1$ for all $j\in \{2,3,4\}$. We
accomplish this by determining explicit formulas for $f([1]^m,j)$ in
the following lemma. 

\begin{lem}\label{lem:1m} We have  $f([1]^m,2)=f([1]^m,4)=1$ and
  $f([1]^m,3)=1-9\frac{\lambda_m}{4}$ for all $m\ge 1$. Therefore $0\le
  f([1]^m,j)\le 1$ for all $j\in \{2,3,4\}$ since $0<\lambda_m\le 1/6$
  for all $m\ge 1$, and $\lim_{m\to\infty}f([1]^m,j)=1$ for all $j\in\{1,2,3,4\}$.
\end{lem}
\begin{proof}
We prove the lemma by
induction.  First, consider $\omega=\emptyset$ in \eqref{eq:fw12};
that is, if $m=1$ then:
\begin{eqnarray*}
  f(1,2)=f(1,4)&=&\alpha_1f(\emptyset,1)+\chi_1f(\emptyset,2)+\chi_1f(\emptyset,3)+\chi_1f(\emptyset,4)+\frac14(1-\alpha_1-3\chi_1)\\
               &=&\alpha_1+\frac14(1-\alpha_1-3\chi_1)\\
  &=&\frac14(1+3\alpha_1-3\chi_1)\\
  &=&\frac14\frac{2\lambda_1-3}{2\lambda_1-1}
\end{eqnarray*}
by Lemma \ref{lem:formulas} and
\begin{eqnarray*}
  f(1,3)&=&\beta_1f(\emptyset,1)+\delta_1f(\emptyset,2)+\delta_1f(\emptyset,3)+\delta_1f(\emptyset,4)+\frac14(1-\beta_1-3\delta_1)\\
  &=&\frac14(1+3\beta_1-3\delta_1)=\frac14 \frac{2\lambda_1-2}{2\lambda_1-1}
\end{eqnarray*}
by Lemma \ref{lem:formulas}.
By plugging in $\lambda_1=1/6$ we obtain that $f(1,2)=f(1,4)=1$ and
$f(1,3)=\frac58=1-\frac94\lambda_1$.
Assume that the claims holds for $m\ge 1$. Then, using the induction
hypothesis,  
\eqref{eq:alpham}, \eqref{eq:chim}, \eqref{eq:fw12} and Lemma \ref{lem:formulas}, we obtain: 
\begin{align*}
  f([1]^{m+1},2)&=\alpha_{m+1}f([1]^m,1)+\chi_{m+1}f([1]^m,2)+\chi_{m+1}f([1]^m,3)+\chi_{m+1}f([1]^m,4)\\
                &+\frac14\bigl( 1-\alpha_{m+1}-3\chi_{m+1} \bigr)\\
  &=(9-42\lambda_{m+1}+36\lambda_{m+1}^2)\chi_{m+1}+\chi_{m+1}+\chi_{m+1}\left(
    1-\frac94\lambda_m \right)+\chi_{m+1}\\
                &-\frac14\bigl( 3R(\lambda_{m+1}) \bigr)\chi_{m+1}\\
  \intertext{which, by factoring out $\chi_{m+1}$ and using the fact that
  $R(\lambda_{m+1})=\lambda_m$,  equals}
  &\chi_{m+1}\bigl( 12-42\lambda_{m+1}+36\lambda_{m+1}^2-3\lambda_m
    \bigr)\\
  \intertext{which by replacing $\lambda_m$ with $R(\lambda_{m+1})$
  equals}
                &\chi_{m+1}\bigl(12-87\lambda_{m+1}+180\lambda_{m+1}^2-108\lambda_{m+1}^3\bigr)\\
  &=\chi_{m+1}\cdot 3\bigl(4-29\lambda_{m+1}+60\lambda_{m+1}^2-36\lambda_{m+1}^3\bigr)=1.
\end{align*}
Hence $f([1]^{m+1},2)=f([1]^{n+1}),4)=1$. Now, by using   the
induction hypothesis, \eqref{eq:betam},  \eqref{eq:deltam}, \eqref{eq:fw13} and Lemma \ref{lem:formulas}, we obtain 
\begin{align*}
  f([1]^{m+1},3)&=\beta_{m+1}f([1]^m,1)+\delta_{m+1}f([1]^m,2)+\delta_{m+1}f([1]^m,3)+\delta_{m+1}f([1]^m,4)\\
                &+\frac14\bigl(1-\beta_1-3\delta_1)\\
                &=\beta_{m+1}+\delta_{m+1}+\delta_{m+1}\left(1-\frac94\lambda_m\right)+\delta_{m+1}\\
                &+\frac14\frac{18\cdot 3\lambda_{m+1}(\lambda_{m+1}-1)(1-2\lambda_{m+1})}{3(1-2\lambda_{m+1})f_2(\lambda_{m+1})}.\\ 
  \intertext{Next, by using the fact that
  $3\lambda_{m+1}(1-2\lambda_{m+1})=R(\lambda_{m+1})/(5-6\lambda_{m+1})$,
  we observe that $ f([1]^{m+1},3)$ equals}
  &\chi_{m+1}\left(6(1-4\lambda_{m+1}+3\lambda_{m+1}^2)+2-3\lambda_{m+1}+(2-3\lambda_{m+1})\bigl(
    1-\frac94\lambda_m \bigr)\right.\\
                &+2-3\lambda_{m+1}+\left. \frac14\frac{18(\lambda_{m+1}-1)\lambda_m}{5-6\lambda_{m+1}}\right)\\
  &=\chi_{m+1}\left( 18\lambda_{m+1}^2-33\lambda_{m+1}+12-\frac94\lambda_m\bigl( 2-3\lambda_{m+1}+\frac{2(\lambda_{m+1}-1)}{5-6\lambda_{m+1}} \bigr) \right).
\end{align*}
We simplify next the last parenthesis in the above expression:
\begin{align*}
  \frac94\lambda_m\bigl(
  2-3\lambda_{m+1}+\frac{2(\lambda_{m+1}-1)}{5-6\lambda_{m+1}}
  \bigr)&=\frac94\lambda_m\frac{12-29\lambda_{m+1}+18\lambda_{m+1}^2}{5-6\lambda_{m+1}}\\
  &=\frac94\lambda_m\frac{8-8\lambda_{m+1}}{5-6\lambda_{m+1}}+\frac94\lambda_m\frac{f_2(\lambda_{m+1})}{5-6\lambda_{m+1}}.
\end{align*}
Replacing $\lambda_m$ with
$R(\lambda_{m+1})=3\lambda_{m+1}(1-2\lambda_{m+1})(5-6\lambda_{m+1})$,
we obtain 
\[
  \frac94\lambda_m\bigl(
  2-3\lambda_{m+1}+\frac{2(\lambda_{m+1}-1)}{5-6\lambda_{m+1}}
  \bigr)=27\lambda_{m+1}(1-2\lambda_{m+1})(1-\lambda_{m+1})+\frac94\lambda_{m+1}3(1-2\lambda_{m+1})f_2(\lambda_{m+1}).
\]
Hence, since $3(1-2\lambda_{m+1})f_2(\lambda_{m+1})=1/\chi_{m+1}$, and
after multiplying through the remaining terms and simplifying, we
obtain 
\begin{align*}
  f([1]^{m+1},3)&=\chi_{m+1}(12-87\lambda_{m+1}+180\lambda_{m+1}^2-108\lambda_{m+1}^3)-\frac94\lambda_{m+1}=1-\frac94\lambda_{m+1}.
\end{align*}
The induction is now complete and the lemma is proved.
\end{proof}
Next, we consider $f([2]^m,j)$, where $j\in\{1,2,3,4\}$. We have already shown that $f([2]^m,2)
= 0$ for all $m\ge0$, so we now prove that 0 $\le f([2]^m,j)\le 1$ for
all $j\in \{1,3,4\}$.

\begin{lem}\label{lem:2m}
We have  $f([2]^m,1)=f([2]^m,3)=0$ and
$f([2]^m,4)=\frac34\lambda_m$ for all $m\ge 1$. Therefore  $0\le
f([2]^m,j)\le 1$ for all $j\in \{1,3,4\}$ since $0<\lambda_m\le 1/6$
for all $m\ge 1$. Moreover, $\lim_{m\to \infty}f([2]^m,i)=0$ for all $i\in\{1,2,3,4\}$.
\end{lem}
\begin{proof}
  We prove the lemma by induction.
First, consider $\omega=\emptyset$ in \eqref{eq:fw21} and
\eqref{eq:fw24}, and, hence, $m=1$. We obtain
\begin{eqnarray*}
  f(2,1)=f(2,3)&=&\chi_1f(\emptyset,1)+\alpha_1f(\emptyset,2)+\chi_1f(\emptyset,3)+\chi_1f(\emptyset,4)+\frac14(1-\alpha_1-3\chi_1)\\
               &=&\chi_1+\frac14(1-\alpha_1-3\chi_1)\\
  &=&\frac14(1-\alpha_1+\chi_1)\\
  &=&\frac14\frac{1-6\lambda_1}{3(1-2\lambda_1)}
\end{eqnarray*}
and
\begin{eqnarray*}
  f(2,4)&=&\delta_1f(\emptyset,1)+\beta_1f(\emptyset,2)+\delta_1f(\emptyset,3)+\delta_1f(\emptyset,4)+\frac14(1-\beta_1-3\delta_1)\\
        &=&\frac14(1-\beta_1+\delta_1)\\
  &=&\frac14 \frac{2-6\lambda_1}{3(1-2\lambda_1)}.
\end{eqnarray*}
Since $\lambda_1=1/6$ it follows that
\begin{eqnarray*}
  f(2,1)=f(2,3)&=&\frac14\frac{1-6\lambda_1}{3(1-2\lambda_1)} = 0  
\end{eqnarray*}
and
\[
f(2,4)=\frac14\frac{1-6\lambda_1}{3(1-2\lambda_1)}=\frac3{24}=\frac34\lambda_1.
\]
Assume now that $f([2]^m,1)=f([2]^m,3)=0$ and that $f([2]^m,4)=\frac34\lambda_m$. We prove that $f([2]^{m+1},1)=f([2]^{m+1},3)=0$ and that $f([2]^{m+1},4)=\frac34\lambda_{m+1}$. 
Using \eqref{eq:fw24} we have
\begin{eqnarray*} 
f([2]^{m+1},4)&=&\delta_{m+1}f([2]^m,1)+\beta_{m+1}f([2]^m,2)+\delta_{m+1}f([2]^m,3)+\delta_{m+1}f([2]^m,4)\\
&+&\frac14(1-\beta_{m+1}-3\delta_{m+1})\\
&=&\frac34\lambda_m\delta_{m+1}+\frac14\frac{18\lambda_{m+1}(\lambda_{m+1}-1)3(1-2\lambda_{m+1})}{3(1-2\lambda_{m+1})f_2(\lambda_{m+1})}.
\end{eqnarray*}
Recall from \eqref{eq:deltam} and \ref{lem:formulas}  that
\[
\delta_{m+1}=\frac{2-3\lambda_{m+1}}{3(1-2\lambda_{m+1})f_2(\lambda_{m+1})}.
\]
Then notice from Lemma \ref{lem:formulas} that
\[
3\lambda_{m+1}(1-2\lambda_{m+1})=\frac{R(\lambda_{m+1})}{5-6\lambda_{m+1}}=\frac{\lambda_m}{5-6\lambda_{m+1}}.
\]
Therefore,
\[
f([2]^{m+1},4)=\frac34\lambda_m\frac{2-3\lambda_{m+1}}{3(1-2\lambda_{m+1})f_2(\lambda_{m+1})}+\frac14\frac{18(\lambda_{m+1}-1)\lambda_m}{3(1-2\lambda_{m+1})f_2(\lambda_{m+1})(5-6\lambda_{m+1})}.
\]
Factoring out $\frac34\lambda_{m}$ and the denominator we obtain 
\begin{eqnarray*}
f([2]^{m+1},4)&=&\frac34\lambda_m\frac1{3(1-2\lambda_{m+1})f_2(\lambda_{m+1})}\left(2-3\lambda_{m+1}+\frac{6(\lambda_{m+1}-1)}{5-6\lambda_{m+1}}\right).
\end{eqnarray*}
We now use the following relationship
\[
2-3\lambda_{m+1}+\frac{6(\lambda_{m+1}-1)}{5-6\lambda_{m+1}}=\frac{f_2(\lambda_{m+1})}{5-6\lambda_{m+1}}.
\]
Therefore,
\[
f([2]^{m+1},4)=\frac34\lambda_m\frac1{3(1-2\lambda_{m+1})f_2(\lambda_{m+1})}\frac{f_2(\lambda_{m+1})}{5-6\lambda_{m+1}}=\frac34\lambda_m\frac1{3(1-2\lambda_{m+1})(5-6\lambda_{m+1})}.
\]
Finally, notice that 
\[
\lambda_m=R(\lambda_{m+1})=3\lambda_{m+1}(1-2\lambda_{m+1})(5-6\lambda_{m+1}).
\]
Therefore,
\[
f([2]^{m+1},4)=\frac34\lambda_{m+1}.
\]
The proof that $f([2]^{m+1},1) = f([2]^{m+1},3) = 0$ follows as below:
\begin{eqnarray*}
f([2]^{m+1},1)&=& \chi_{m+1}f([2]^m,1)+\alpha_{m+1}f([2]^m,2)+\chi_{m+1}f([2]^m,3)+\chi_{m+1}f([2]^m,4)\\
&+&\frac14(1-\alpha_{m+1}-3\chi_{m+1})\\
&=&\frac34\chi_{m+1}\lambda_m-\frac14\frac{3R(\lambda_{m+1})}{3(1-2\lambda_{m+1})(f_2(\lambda_{m+1}))}\\
&=&\frac14\frac{3\lambda_m-3R(\lambda_{m+1})}{3(1-2\lambda_{m+1}f_2(\lambda_{m+1})}\\
&=& \frac14\frac{3\lambda_m-3\lambda_m}{3(1-2\lambda_{m+1})f_2(\lambda_{m+1})} = 0.
\end{eqnarray*}
The induction is now complete and the lemma is proved.
\end{proof}
We turn now to prove  $0\le f([5]^m,j)\le 1$ for all $j\in
\{1,2,3,4\}$. Recall from Lemmas \ref{lem:fw},  \ref{lem:1m} and
 \ref{lem:2m} that 
$f(5,1)=f(1,3)=1-\frac94\lambda_1$ and $f(5,2)=f(5,3)=f(5,4)=f(2,4)=\frac34\lambda_1$.
\begin{lem}\label{lem:5m}
  We have 
  \begin{equation}\label{eq:5m1}
    f([5]^m,1)=\frac14+\frac14\frac1{3^{m-1}}\prod_{k=1}^m\frac{1}{1-2\lambda_k}
  \end{equation}
  and
  \begin{equation}\label{eq:5m2}
    f([5]^m,2)=f([5]^m,3)=f([5]^m,4)=\frac14-\frac14\frac1{3^m}\prod_{k=1}^m\frac1{1-2\lambda_k}
  \end{equation}
  for all $m\ge 2$. Therefore, $\{f([5]^m,1)\}_{m\ge 1}$ is a decreasing sequence
  and $\{f([5]^m,2)\}_{m\ge 1}$ is an increasing sequence. The limit
  of both of these sequences equals $1/4$.
\end{lem}
\begin{proof}
  We begin with $m=2$. Using 
  \eqref{eq:fw13} we have 
  \begin{align*}
    f(55,1)&=f(51,3)=\beta_2f(5,1)+\delta_2f(5,2)+\delta_2f(5,3)+\delta_3f(5,4)\\
           &+\frac14\bigl( 1-\beta_2-3\delta_2 \bigr)\\
    &=\beta_2\left( 1-\frac94\lambda_1
      \right)+3\delta_2\frac34\lambda_1+\frac14\bigl(
      1-\beta_2-3\delta_2 \bigr)\\
           &=\frac14+\frac34(\beta_2-\delta_2)-\frac94\lambda_1(\beta_2-\delta_2)\\
&=\frac14+\frac34(\beta_2-\delta_2)(1-3\lambda_1)\\
    &=\frac34\frac1{3(1-2\lambda_2)}(1-3\lambda_1)\\
&=\frac14+\frac14\frac13\frac1{1-2\lambda_1}\frac1{1-2\lambda_2},
  \end{align*}
  where in the last step we used Lemma \ref{lem:formulas} and the fact that
  $1-3\lambda_1=\frac13\frac1{1-2\lambda_1}$ since $\lambda_1=1/6$.

  Using equation \eqref{eq:fw24} we obtain
  \begin{align*}
    f(55,2)=f(52,4)&=\delta_2f(5,1)+\beta_2f(5,2)+\delta_2f(5,3)+\delta_2f(5,4)\\
           &+\frac14\bigl( 1-\beta_2-3\delta_2 \bigr)\\
    &=\delta_2\left( 1-\frac94\lambda_1
      \right)+\beta_2\frac34\lambda_1+2\delta_2\frac34\lambda_1+\frac14\bigl(
      1-\beta_2-3\delta_2 \bigr)\\
    &=\frac14-\frac14(\beta_2-\delta_2)(1-3\lambda_1)=\frac14-\frac14\frac1{3^2}\frac1{1-2\lambda_1}\frac1{1-2\lambda_2}.
  \end{align*}
  Since $f(5,2)=f(5,3)=f(5,4)$, the equations \eqref{eq:fw24},
  \eqref{eq:fw31}, and \eqref{eq:fw42} imply that $f(55,2)=f(55,3)=f(55,4)$.
  Assume now that \eqref{eq:5m1} and \eqref{eq:5m2} hold for $m\ge
  2$. We prove the induction step:
  \begin{align*}
    f([5]^{m+1},1)=f([5]^m1,3)&=\beta_{m+1}f([5]^m,1)+\delta_{m+1}f([5]^m,2)+\delta_{m+1}f([5]^m,3)+\delta_{m+1}f([5]^m,4)\\
                  &+\frac14\bigl( 1-\beta_{m+1}-3\delta_{m+1} \bigr)\\
    &=\beta_{m+1}\left(
      \frac14+\frac14\frac1{3^{m-1}}\prod_{k=1}^m\frac{1}{1-2\lambda_k}
      \right)+3\delta_{m+1}\left(
      \frac14-\frac14\frac1{3^m}\prod_{k=1}^m\frac1{1-2\lambda_k}
      \right)\\
                  &+\frac14\bigl( 1-\beta_{m+1}-3\delta_{m+1} \bigr)\\
    \intertext{which, after we cancel out $\frac14\beta_{m+1}$ and
    $\frac34\delta_{m+1}$,  and factor out the products (notice that
    the 3 in front of $\delta_{m+1}$ reduces the power of 3 in the
    second product), equals}
    &\frac14+\frac14\frac1{3^{m-1}}\prod_{k=1}^m\frac1{1-2\lambda_k}\bigl(
      \beta_{m+1}-\delta_{m+1} \bigr)
                  =\frac14+\frac14\frac1{3^m}\prod_{k=1}^{m+1}\frac1{1-2\lambda_k},               
  \end{align*}
  where we used Lemma \ref{lem:formulas}  in the last
  equality.

  Using now equation \eqref{eq:fw24} we obtain 
  \begin{align*}
    f([5]^{m+1},1)=f([5]^m1,3)&=\delta_{m+1}f([5]^m,1)+\beta_{m+1}f([5]^m,2)+\delta_{m+1}f([5]^m,3)++\delta_{m+1}f([5]^m,4)\\
                  &+\frac14\bigl( 1-\beta_{m+1}-3\delta_{m+1} \bigr)\\
    &=\delta_{m+1}\left(
      \frac14+\frac14\frac1{3^{m-1}}\prod_{k=1}^m\frac{1}{1-2\lambda_k}
      \right)+\beta_{m+1}\left(
      \frac14-\frac14\frac1{3^m}\prod_{k=1}^m\frac1{1-2\lambda_k}
      \right)\\
    &+2\delta_{m+1}\left(
      \frac14-\frac14\frac1{3^m}\prod_{k=1}^m\frac1{1-2\lambda_k}
      \right)+\frac14\bigl( 1-\beta_{m+1}-3\delta_{m+1} \bigr)\\
    \intertext{which, after canceling  $\frac14\beta_{m+1}$ and
    $\frac34\delta_{m+1}$, and simplifying the remaining two
    expressions involving $\delta_{m+1}$ (notice that the expression
    with a $+$ in front of the big product needs to be multiplied by a
    3 for the common denominator) equals}
                  &\frac14+\delta_{m+1}\frac14\frac1{3^m}\prod_{k=1}^m\frac1{1-2\lambda_k}-\beta_{m+1}\frac14\frac1{3^m}\prod_{k=1}^m\frac1{1-2\lambda_k}\\
    &=\frac14-\frac14\frac1{3^m}\prod_{k=1}^m\frac1{1-2\lambda_k}\bigl(
      \beta_{m+1}-\delta_{m+1} \bigr)\\
    &=\frac14-\frac14\frac1{3^{m+1}}\prod_{k=1}^{m+1}\frac1{1-2\lambda_k}
  \end{align*}
  by using Lemma \ref{lem:formulas} again.
  The induction is complete. The last part of the lemma follows from
  the fact that $\{\lambda_m\}$ is a decreasing sequence whose limit
  is 0 and $\lambda_1=1/6$. In particular,
  $\lim_{m\to\infty}\frac1{3^m}\prod_{k=1}^m\frac{1}{1-2\lambda_k}$ is
  decreasing to 0.
\end{proof}
Thus, 0 $\le f([i]^m,j)\le 1$ for all $i\in\{1,2,3,4,5\}$ and $j\in \{1,2,3,4\}$. 
The next step is to prove that $0\le f(\omega,j)\le 1$ for all
$(\omega,j)\in \overline{\Sigma^*}$. We accomplish this by showing
that, for a fixed $i\in \Sigma$, we have 
$\max_{(i\omega,j)}f(i\omega,j)=\max_{j}f(i,j)$ and
$\min_{(i\omega,j)}f(i\omega,j)=\min_{j}f(i,j)$ (where $i\omega$ is
the word formed by the letter $i$ followed by the word $\omega$). 
By Remark \ref{rem:rot} we only need to prove the statement for
$i=1$, $i=2$ and $i=5$.
\begin{prop}\label{prop:maxf}
  With the notation as above, we have 
  \[\max_{\{(i\omega,j)\,:\,\omega\in
      \Sigma^*,j\in\{1,2,3,4\}\}}f(i\omega,j)=\max_{j\in\{1,2,3,4\}}f(i,j)\]
  and
\[\min_{\{(i\omega,j)\,:\,\omega\in
      \Sigma^*,j\in\{1,2,3,4\}\}}f(i\omega,j)=\min_{j\in\{1,2,3,4\}}f(i,j)\] for all $i\in\{1,2,3,4,5\}$. Therefore,
$\max_{(\omega,j)\in\overline{\Sigma^*}}f(\omega,j)=1$ and
$\min_{(\omega,j)\in \overline{\Sigma^*}}f(\omega,j)=0$.
\end{prop}
\begin{proof}
  We begin by proving the proposition for the case
  $\vert\omega\vert=1$. 
  As noted above, it suffices to consider the cases  $i=1,2$ and
  $5$. We begin with $i=1$. Lemma \ref{lem:1m}  implies that
  $\min_{j\in\{1,2,3,4\}}f(1,j)=1-9\lambda_1/4$ and
  $\max_{j\in\{1,2,3,4\}}f(1,j)=1$. Moreover, the minimum is attained
  at $j=3$, and the maximum is attained at $j=1,2$ and 4. The second part of Lemma
  \ref{lem:w234} with $\omega=1$ and Lemma \ref{lem:1m} implies that 
  \[
    f(11,2)=f(11,4)=f(12,1)=f(12,3)=f(14,1)=f(14,3)=1
  \]
  and $f(11,3)=f(12,4)=f(14,2)=1-9\lambda_2/4$. Recall also that
  $f(11,1)=f(1,1)=1$, $f(12,2)=f(1,2)=1$, $f(13,3)=f(1,3)=1$,
  $f(14,4)=f(1,4)=1-\frac94\lambda_1$, $f(15,1)=f(11,3)$, $f(15,2)=f(12,4)$,
  $f(15,3)=f(13,1)$, and $f(15,4)=f(14,2)$. Therefore we only need
  to check that $1-9/4\lambda_1\le f(13,2)=f(13,4),f(13,1)\le
  1$. Using equation \eqref{eq:fw32} we have
  \begin{align*}
    f(13,2)&=\chi_2f(1,1)+\chi_2f(1,2)+\alpha_2f(1,3)+\chi_2f(1,4)\\
           &+\frac14\bigl( 1-\alpha_2-3\chi_2 \bigr)\\
    &=\chi_2+\chi_2+\alpha_2\bigl( 1-\frac94\lambda_1 \bigr)+\chi_2-\frac34R(\lambda_2)\chi_2.
  \end{align*}
  Using  \eqref{eq:alpham} we have 
  \[
    \alpha_2=(9-42\lambda_2+36\lambda_2^2)\chi_2=\chi_2+(8-42\lambda_2+36\lambda_2^2)\chi_2.
  \]
  Therefore,
  \begin{align*}
    \alpha_2\bigl( 1-\frac94\lambda_1 \bigr)&=\chi_2\bigl(
                                              1-\frac94\lambda_1
                                              \bigr)+(8-42\lambda_2+36\lambda_2^2)\chi_2\bigl(
                                              1-\frac94\lambda_1
                                              \bigr)\\
    &< \chi_2\bigl( 1-\frac94\lambda_2 \bigr)+(8-42\lambda_2+36\lambda_2^2)\chi_2
  \end{align*}
  since $1-9\lambda_1/4<1$. Therefore,
  \begin{align*}
    f(13,2)&<\chi_2+\chi_2+\chi_2\bigl( 1-\frac94\lambda_1
             \bigr)+(8-42\lambda_2+36\lambda_2^2)\chi_2+\chi_2-\frac34\lambda_1\chi_2\\
           &=\chi_2+\chi_2+\chi_2\bigl(1-\frac94\lambda_1\bigr)+\alpha_2-\frac34\lambda_1\chi_2\\
    &=1,
  \end{align*}
  where we grouped a $\chi_2$ with
  $(8-42\lambda_2+36\lambda_2^2)\chi_2$ to obtain $\alpha_2$, and we
  used the proof of Lemma \ref{lem:1m} to get the last equality. 
  A similar computation shows that
  \begin{align*}
    f(13,4)&=\chi_2+\chi_2+\chi_2\bigl( 1-\frac94\lambda_1
             \bigr)+\alpha_2-(8-42\lambda_2+36\lambda_2^2)\chi_2\bigl(
             1-\frac94\lambda_1 \bigr)-\frac34\lambda_1\chi_2\\
           &=1-(8-42\lambda_2+36\lambda_2^2)\chi_2\frac94\lambda_1\\
    &>1-\frac94\lambda_1.
  \end{align*}
  For the last  remaining vertex, we  use \eqref{eq:fw13}. We have
  \begin{align*}
    f(13,1)&=\delta_2f(1,1)+\delta_2f(1,2)+\beta_2f(1,3)+\delta_2f(1,4)\\
           &+\frac14(1-\beta_2-3\delta_2)\\
    &=\delta_2+\delta_2+\beta_2\bigl( 1-\frac94\lambda_1 \bigr)  +\delta_2+\frac14(1-\beta_2-3\delta_2).
  \end{align*}
  Since $\beta_2-\delta_2=\frac1{3-6\lambda_2}$, we have 
  $\beta_2=\delta_2+\frac1{3-6\lambda_2}$. Therefore,
  \begin{align*}
    f(13,1)&=\delta_2+\delta_2+\delta_2\bigl( 1-\frac94\lambda_1
             \bigr)+\frac1{3-6\lambda_2}\bigl( 1-\frac94\lambda_1 \bigr)+\delta_2+\frac14(1-\beta_2-3\delta_2)\\
    &<\delta_2+\delta_2+\delta_2\bigl( 1-\frac94\lambda_1
      \bigr)+\frac1{3-6\lambda_2}+\delta_2+\frac14(1-\beta_2-3\delta_2)\\
    \intertext{which, by combining a $\delta_2$ with
    $\frac1{3-6\lambda_2}$, equals}
    &\beta_2+\delta_2+\delta_2+\delta_2\bigl( 1-\frac94\lambda_1
      \bigr)+\frac14(1-\beta_2-3\delta_2)\\
    &=1-\frac92\lambda_2<1,
  \end{align*}
  by the proof of Lemma \ref{lem:1m}. To prove that $f(13,1)>f(1,3)$
  we modify the above proof as follows:
  \begin{align*}
    f(13,1)-f(1,3)&=\delta_2+\delta_2+\bigl(1-\frac94\lambda_1
                    \bigr)(\beta_2-1)+\delta_2+\frac14(1-\beta_2-3\delta_2)
  \end{align*}
  which using that     $\beta_2=\delta_2+\frac1{3-6\lambda_2}$ equals
  \begin{align*}
    &\bigl( 1-\frac94\lambda_1 \bigr)\delta_2+\bigl(
    1-\frac94\lambda_1 \bigr)\bigl( \frac1{3-6\lambda_2}-1
      \bigr)+(\delta_2+\frac1{3-6\lambda_2})-\frac1{3-6\lambda_2}+\delta_2+\delta_2\\
    &+\frac14(1-\beta_2-3\delta_2).\\
    \intertext{Since $\beta_2=\delta_2+\frac1{3-6\lambda_2}$,  by the
    proof of Lemma \ref{lem:1m}, and by 
    factoring out $\frac1{3-6\lambda_2}$, $f(13,1)-f(1,3)$ equals}
    &1-\frac94\lambda_2-\bigl( 1-\frac94\lambda_1 \bigr)+\bigl(
      1-\frac94\lambda_1-1 \bigr)\frac1{3-6\lambda_2}\\
    &>-\frac94\lambda_2+\frac94\lambda_1-\frac34\lambda_1=-\frac94\lambda_2+\frac64\lambda_1>0,
  \end{align*}
  where we used the fact that $\frac1{3-6\lambda_2}>\frac13$ and the
  fact that the inequality $\lambda_1>\frac32\lambda_2$ is clearly
  true. Therefore $f(13,1)>f(1,3)=1-\frac94\lambda_1$. 
    
  We move now to the case $i=2$ and $\vert\omega\vert=1$. We note that
  by Lemma \ref{lem:2m}, $\min_{j\in\{1,2,3,4\}}f(2,j)=0$ 
  and $\max_{j\in\{1,2,3,4\}}f(2,j)=3\lambda_1/4$. Moreover, the
  minimum is attained at $j=1,2$ and 3, and the maximum is attained
  at $j=4$.
  Therefore Lemma \ref{lem:2m} and the first part
  of Lemma \ref{lem:w234} applied to $\omega=2$  imply that
  \[
    f(22,1)=f(22,3)=f(23,2)=f(23,4)=f(21,2)=f(21,4)=0,
  \]
  and $f(21,3)=f(22,4)=f(23,1)=3\lambda_2/4<3\lambda_1/4$. Also,
  $f(22,2)=f(2,2)=0$, $f(23,3)=f(2,3)=0$, 
  $f(24,4)=f(2,4)=3\lambda_1/4$, $f(25,1)=f(21,3)$, $f(25,2)=f(22,4)$,
  $f(25,3)=f(23,1)$, and $f(25,4)=f(24,2)$. Therefore, we only need to check
  that $0<f(24,1)=f(24,3),f(24,2)<3/4\lambda_1$. We have 
  \begin{align*}
    f(24,1)&=\chi_2f(2,1)+\chi_2f(2,2)+\chi_2f(2,3)+\alpha_2f(2,4)+\frac14\bigl(
             1-\alpha_2-3\chi_2 \bigr)\\ 
    &=\alpha_2f(2,4)+\frac14\bigl( 1-\alpha_2-3\chi_2 \bigr).
  \end{align*}
  It follows that
  \[
    f(24,1)-f(2,4)=f(2,4)(\alpha_2-1)+\frac14(1-\alpha_2-3\chi_2)<0
  \]
  since $\alpha_2-1<0$ and $\frac14\bigl( 1-\alpha_2-3\chi_2 \bigr)<0$
  (see Lemma \ref{lem:formulas}). Hence
  $f(24,1)<f(2,4)=3\lambda_1/4$. Moreover, since $\alpha_2>\chi_2$
  (see Lemma \ref{lem:formulas}), we have 
  $f(24,2)>f(21,3)=3\lambda_2/4>0$ by comparing \eqref{eq:fw43}
  against \eqref{eq:fw13}.

  Using now equation \eqref{eq:fw42}, we obtain 
  \[
    f(24,2)=\beta_2f(2,4)+\frac14(1-\beta_2-3\delta_2).
  \]
  Hence
  \[
    f(24,2)-f(2,4)=f(2,4)(\beta_2-1)+\frac14(1-\beta_2-3\delta_2)<0
  \]
  since $\beta_2<1$ and $1-\beta_2-3\delta_2<0$ (see Lemma
  \ref{lem:formulas}). Hence $f(24,2)<f(2,4)=3\lambda_1/4$. Moreover,
  since $\beta_2>\delta_2$ it follows that
  $f(24,2)>f(21,3)=3\lambda_2/4>0$ by comparing \eqref{eq:fw42}
  against \eqref{eq:fw13}. So we are done with $i=2$ and
  $\vert\omega\vert=1$.

  Next we consider $i=5$ and $\vert\omega\vert=1$. We have 
  $\min_{j\in\{1,2,3,4\}}f(5,j)=\frac34\lambda_1$ and\\
  $\max_{j\in\{1,2,3,4\}}f(5,j)=1-\frac94\lambda_1$. Moreover, the
  maximum is attained at $j=1$ and the minimum is attained at $j=2,3$
  and $4$. Therefore, using the last part of Lemma \ref{lem:w234} we
  obtain 
  \[
    f(52,1)=f(52,3)=f(53,2)=f(53,4)=f(54,1)=f(54,3)
  \]
  and $f(52,4)=f(53,1)=f(54,2)$. Moreover, $f(52,4)=f(55,2)$,
  $f(53,1)=f(55,3)$, and $f(54,2)=f(55,4)$; all of these values are
  given by \eqref{eq:5m2} with $m=2$. We also know the value of
  $f(55,1)=f(51,3)$ from \eqref{eq:5m1}. Therefore, we only need to
  check that the  value of $f(52,1)$ 
   is between $3\lambda_1/4$ and $1-9\lambda_1/4$. Using
  \eqref{eq:fw21} we have 
  \[
    f(52,1)=\chi_2f(5,1)+\alpha_2f(5,2)+\chi_2f(5,3)+\chi_2f(5,4)+\frac14(1-\alpha-3\chi_2).
  \]
  Therefore,
  \begin{align*}
    f(52,1)-f(5,2)&=\chi_2f(5,1)+(\alpha_2-1)f(5,2)+\chi_2f(5,3)+\chi_2f(5,4)+\frac14(1-\alpha_2-3\chi_2)\\
    \intertext{which, since $\alpha_2-1=3(R(\lambda_2)-1)\chi_2$ (see
    Lemma \ref{lem:formulas}), equals}
    &\chi_2\bigl( 1-\frac94\lambda_1
      \bigr)+3(\lambda_1-1)\chi_2\frac34\lambda_1+\chi_2\frac34\lambda_1+\chi_2\frac34\lambda_1-\frac34\lambda_1\chi_2\\
    &=\chi_2\left( 1-\frac{15}4\lambda_1+\frac94\lambda_1^2 \right)>0
  \end{align*}
   by Lemma \ref{lem:formulas}. Hence $f(52,1)>f(5,2)$. To prove that
   $f(52,1)$ is smaller than $f(5,1)$ we proceed as follows using \eqref{eq:fw21}:
   \begin{align*}
     f(52,1)-f(5,1)&=(\chi_2-1)f(5,1)+\alpha_2f(5,2)+\chi_2f(5,3)+\chi_2f(5,4)+\frac14(1-\alpha_2-3\chi_2)\\
     &=(\chi_2-1)\bigl( 1-\frac94\lambda_1
       \bigr)+\alpha_2\frac34\lambda_1+\chi_2\frac34\lambda_1+\chi_2\frac34\lambda_1-\frac34\lambda_1\chi_2\\
     &=\chi_2\left( 1-\frac{15}4\lambda_1+\frac94\lambda_1^2 \right)-1+\frac94\lambda_1<0
   \end{align*}
   since, by Lemma \ref{lem:formulas}, $\chi_2\left(
     1-\frac{15}4\lambda_1+\frac94\lambda_1^2 \right)<0.22$ and
   $1-\frac94\lambda_1=0.625$ (for $n\ge 3$ we have 
   $1-\frac94\lambda_n>0.625$). Therefore $f(52,1)<f(5,1)$, and so
    we proved the statement of the Proposition for
   $\vert\omega\vert=1$.

   Moving to $\vert\omega\vert>1$, 
   we see that we can repeat the above arguments inductively when increasing the
   length of
   $\vert\omega\vert$ from $m$ to $m+1$. This is because, by
   Lemma \ref{lem:fw}, $f(i\omega k,j)$ depends only on the values
   $f(i\omega,1), f(i\omega,2), f(i\omega,3)$ and $f(i\omega,4)$. If
   $i=1$, then the possible combinations of these values are (see the
   computations above): three of
   them are equal to 1 and the fourth equals $1-\frac94\lambda_m$ (by
   Lemma \ref{lem:1m}); one value is $1-\frac94\lambda_{m-1}$, one
   value is $1-\frac94\lambda_m$, and the other two values are between
   $1-\frac94\lambda_{m-1}$ and 1; and
   three values equal $\frac34\lambda_m$ and one value equals
   $1-\frac94\lambda_m$. Then the proof given for $i=1$ and
   $\vert\omega\vert=1$ can be easily adapted to prove the inductive
   step.

   If $i=2$, the possible values of $f(2\omega,1), f(2\omega,2),
   f(2\omega,3)$ and $f(2\omega,4)$ are: three of them equal 0, and
   one equals $3\lambda_m/4$; three of them equal $3\lambda_m/4$ and
   the fourth one is between $3\lambda_m/4$ and $3\lambda_{m+1}/4$;
   and one of the values is $3\lambda_i/4$ for some $2\le i\le m-1$,
   and the other three values are between $3\lambda_m/4$ and $3\lambda_i/4$.
   Therefore, the arguments given above for $i=2$ and
   $\vert\omega\vert=1$ can be easily adapted to these cases to prove
   the inductive step.

   If $i=5$ then the possible values of $f(5\omega,1), f(5\omega,2),
   f(5\omega,3)$ and $f(5\omega,4)$: all four are given by Lemma
   \ref{lem:5m} and, in particular, three of them are equal; and three of
   them are equal and all of them are in between the values provided
   by Lemma \ref{lem:5m}.
   Then the above argument for $i=5$ and $\vert\omega\vert=1$  can
   also be adapted for the induction step on the length of $\omega$. 

   The last statement of Proposition \ref{prop:maxf} follows immediately since
   \[\max_{i\in\{1,2,3,4,5\}}\max_{j\in\{1,2,3,4\}}f(i,j)=1\] and
   \[\min_{i\in\{1,2,3,4,5\}}\min_{j\in\{1,2,3,4\}}f(i,j)=0.\] 
\end{proof}

\begin{proof}[Proof of Lemma \ref{lem:main}]
The  statement of Proposition \ref{prop:maxf} is true for $g,h$ and $k$ because
$g(\omega,i),$ $h(\omega,i)$ and $k(\omega,i)$ can be obtained from $f$ by
shifting the letters. For example,
$g(\omega,i)=f(R_2(\omega),R_2(i))$, where $R_2$ is the ``rotation''
defined in Remark \ref{rem:rot}, and similar formulas hold for $h$ and
$k$.  Therefore
$0\le f(\omega,i), g(\omega,i),h(\omega,i),k(\omega,i) \le 1$ for
every $\omega \in  \Sigma^*$ and for every $i \in \{1,2,3,4\}$. 
\end{proof}

\appendix

\section{A few facts used in the proofs}

We now collect several relationships and formulas satisfied by
$\alpha,\beta,\chi$ and $\delta$ that we used throughout the
paper. The proofs of the following statements are  straight
computations and/or easy calculus problems.
 \begin{lem}\label{lem:formulas}
  \begin{enumerate}
\item The map $R(\lambda)$ can be factored out as 
 \[
     R(\lambda)=36\lambda^3-48\lambda^2+15\lambda=3\lambda(1-2\lambda)(5-6\lambda),
     \]
and it satisfies the relation
   \[
     3R(\lambda)-4=(6\lambda-1)f_2(\lambda).
     \]
\item The following identities are true for all $\lambda$ that are not
  forbidden values:
\begin{align*}
  \gamma=\chi&=\frac{1}{3(4-29\lambda+60\lambda^2-36\lambda^3)}=\frac{1}{3(1-2\lambda)f_2(\lambda)},\\
  1+3\alpha-3\chi&=\frac{2\lambda-3}{2\lambda-1},\\
  1+3\beta-3\delta&=\frac{2\lambda-2}{2\lambda-1},\\
  1-\beta-3\delta&=\frac{18\lambda(\lambda-1)3(1-2\lambda)}{3(1-2\lambda)f_2(\lambda)}=\frac{18(\lambda-1)R(\lambda)}{3(1-2\lambda)f_2(\lambda)(5-6\lambda)},\\
  1-\alpha-3\chi&=-\frac{3\lambda(5-6\lambda)3(1-2\lambda)}{3(1-2\lambda)f_2(\lambda)}=-\frac{3R(\lambda)}{3(1-2\lambda)f_2(\lambda)},\\
  \beta-\delta&=\frac{1}{3-6\lambda},\\
  \alpha-1&=3(R(\lambda)-1)\chi_2.
\end{align*}
\item $0<\alpha<1$ and $0<\beta<1$ for all $0<\lambda<\lambda_2$, and $0<\delta<1$ for all
  $0<\lambda<1/6=\lambda_1$. In particular,
  $0<\alpha_n,\beta_n,\delta_n<1$ for all $n\ge 2$.
\item $1-\alpha-3\chi<0$ and $1-\beta-3\delta<0$.
\item $\beta>\delta$ and $\alpha>\chi$ for all
  $0<\lambda<1/6$. Therefore $\beta_n>\delta_n$ and $\alpha_n>\chi_n$
  for all $n\ge 2$.
\item $0.08<\chi\bigl( 1-\frac{15}4\lambda+\frac94\lambda^2 \bigr)<0.22$
  for all $0<\lambda<=1/6$.
\end{enumerate}
\end{lem}

\begin{proof}
  \begin{enumerate}
  \item  The formula for $f_2$ is proved in \cite{ConStrWhe} and follows
    immediately from its definition.    
  \item These formulas were also discussed and used in
    \cite{ConStrWhe}.
  \item These formulas follow by direct computations. We also used the
    Maxima CAS \cite{maxima} to double check our computations. The Maxima code that
    we used is provided on our website.
  \item This can be shown using standard methods of calculus. We provide next
    pictures that illustrate our claim:
      \begin{figure*}[ht]
    \centering
    \begin{subfigure}[b]{0.3\textwidth}
        \includegraphics[width=\textwidth]{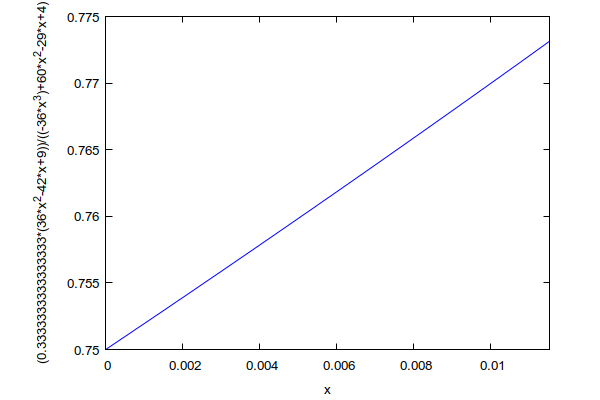}
        \caption{$\alpha$ with $0\le\lambda\le \lambda_2$.}
    \end{subfigure}%
 ~
    \begin{subfigure}[b]{0.3\textwidth}
        \includegraphics[width=\textwidth]{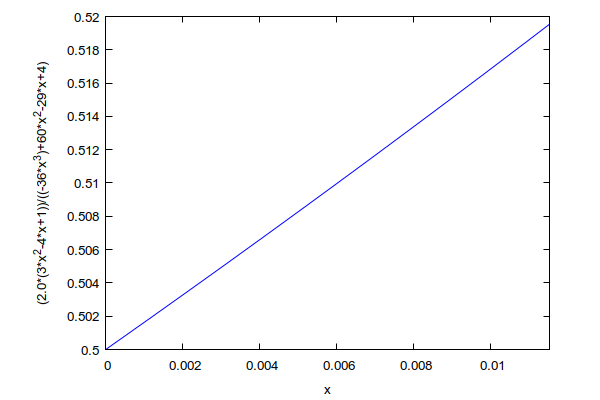}
        \caption{$\alpha$ with $0\le\lambda\le \lambda_2$.}       
      \end{subfigure}
      ~
       \begin{subfigure}[b]{0.3\textwidth}
        \includegraphics[width=\textwidth]{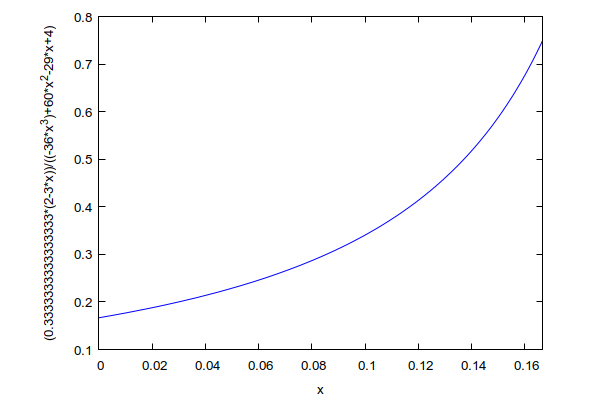}
        \caption{$\delta$ with $0\le\lambda\le \lambda_1$.}
    \end{subfigure}
\end{figure*}
   \item Once again we provide a picture instead of presenting the
     computation of the derivative of each expression:
           \begin{figure*}[ht]
    \centering
    \begin{subfigure}[b]{0.35\textwidth}
        \includegraphics[width=\textwidth]{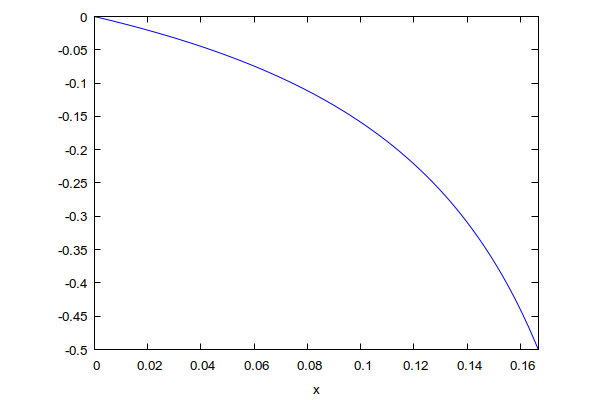}
        \caption{$\frac14(1-\alpha-3\chi)$ with $0\le\lambda\le \lambda_1$.}
    \end{subfigure}%
 ~
    \begin{subfigure}[b]{0.35\textwidth}
        \includegraphics[width=\textwidth]{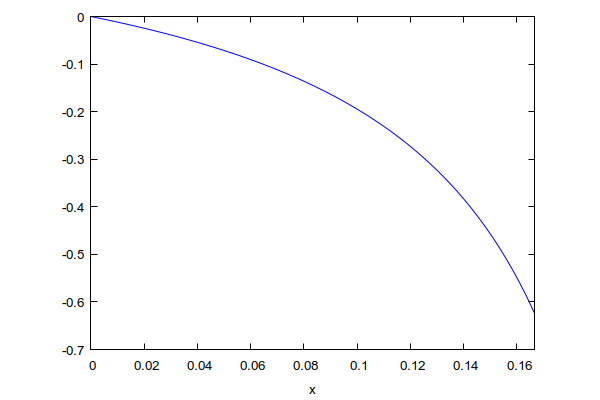}
        \caption{$\frac14(1-\beta-3\delta)$ with $0\le\lambda\le \lambda_1$.}       
      \end{subfigure}
    \end{figure*}
  \item One more picture to illustrate the last statement:
    \begin{figure*}[ht]
      \centering
      \includegraphics[scale=0.35]{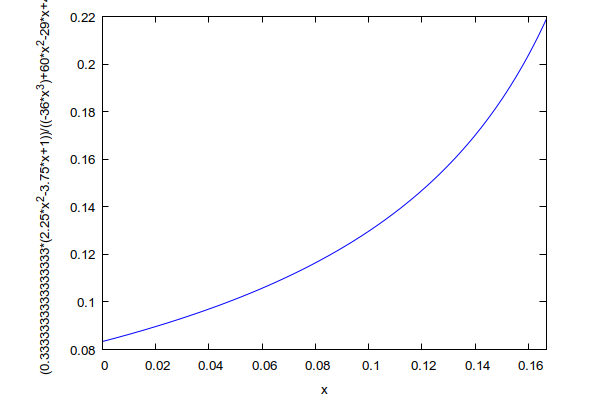}
      \caption{$\chi\bigl(1-\frac{15}4\lambda+\frac94\lambda^2\bigr)$
        with $0\le \lambda\le \lambda_1$.}
    \end{figure*}
  \end{enumerate}
\end{proof}
\bibliographystyle{unsrt}
\bibliography{Vicsek}

\end{document}